\documentclass[a4paper,11pt]{elsarticle}

\usepackage{amsfonts}
\usepackage{amsthm}
\usepackage{amssymb, amsmath}
\usepackage[english]{babel}

\newtheorem{theo}{Theorem}[section]

\newtheorem{lemma}[theo]{Lemma}

\theoremstyle{definition} 

\newtheorem{remark}[theo]{Remark}
\newtheorem{definition}[theo]{Definition}
\usepackage{enumitem}
\newcommand{\p}{\vspace{0.3cm}\\}

\newcommand{\R}{\mathbb{R}}

\newcommand{\N}{\mathbb{N}}

\newcommand{\eps}{\varepsilon}
\newcommand{\ph}{\varphi}

\newcommand{\err}{\mathrm{err}}

\newcommand{\sgn}{\mathrm{sign}}

\numberwithin{equation}{section}
\DeclareMathOperator*\argmin{arg\,min}

\begin{document}
\title{On some aspects of approximation of ridge functions}
\author[tu]{Anton Kolleck}
\ead{kolleck@math.tu-berlin.de}
\author[tu]{Jan Vyb\'\i ral\corref{cor1}}
\ead{vybiral@math.tu-berlin.de}

\cortext[cor1]{Corresponding author}
\address[tu]{Mathematical Institute, Technical University Berlin, Strasse des 17. Juni 136, D-10623 Berlin, Germany}

\begin{abstract} We present effective algorithms for uniform approximation of multivariate functions satisfying some prescribed inner structure.
We extend in several directions the analysis of recovery of ridge functions $f(x)=g(\langle a,x\rangle)$ as performed earlier by one of the authors and his coauthors.
We consider ridge functions defined on the unit cube $[-1,1]^d$ as well as recovery of ridge functions defined on the unit ball from noisy measurements.
We conclude with the study of functions of the type $f(x)=g(\|a-x\|_{l_2^d}^2)$.
\end{abstract}


\maketitle

\section{Introduction}

Functions depending on a large number of variables play nowadays a crucial role in many areas, including
parametric and stochastic PDE's, bioinformatics, financial mathematics, data analysis and learning theory.
Together with an extensive computational power being used in these applications, results on basic numerical aspects
of these functions became crucial. Unfortunately, multivariate problems suffer often from the \emph{curse of dimension},
i.e. the minimal number of operations necessary to achieve (an approximation of) a solution grows exponentially
with the underlying dimension of the problem. Although this effect was observed many times in the literature,
we refer to \cite{NW_2009_2} for probably the most impressive result of this kind - namely that even the uniform approximation
of infinitely-differentiable functions is intractable in high dimensions.

In the area of \emph{Information Based Complexity} it was possible to achieve a number of positive results on tractability
of multivariate problems by posing an additional (structural) assumption on the functions under study. The best studied concepts
in this area include tensor product constructions and different concepts of anisotropy and weights. We refer to the series of monographs
\cite{NW1,NW2,NW3} for an extensive treatment of these and related problems. 
We pursue the direction initiated by Cohen, Daubechies, DeVore, Kerkyacharian and Picard in \cite{paper1} and further developed
in a series of recent papers \cite{paper3, HC, MUV}.  This line of study is devoted to \emph{ridge functions}, which
are multivariate function $f$ taking the form $f(x)=g(\langle a,x\rangle)$ for some univariate function $g$ and a non-zero vector $a\in\R^d.$ We refer also to \cite{DPW, SV, WW} for a related approach.

Functions of this type are by no means new in mathematics. 

They appear for example very often in statistics in the frame of the so-called \emph{single index models}.
They play also an important role in approximation theory, where their simple structure
motivated the question if a general function could be well approximated by sums of ridge functions. 
The pioneering work in this field is \cite{LS}, where the term ``ridge function'' was first introduced, and also \cite{LP}, where the fundamentality of ridge functions was investigated.
Ridge functions appeared also in mathematical analysis of neural networks \cite{C1,P2}
and as the basic building blocks of \emph{ridgelets} of Cand\`es and Donoho \cite{CD}.
A survey on approximation by (sums of) ridge functions was given in \cite{P1}.

The biggest difference between our setting and the usual approach of statistical learning
and data analysis is that we suppose that the sampling points of $f$ can be freely chosen,
and are not given in advance. This happens, for instance, if sampling of the unknown function at a point is
realized by a (costly) PDE solver.

Most of the techniques applied so far in recovery of ridge functions are based on the simple formula
\begin{equation}\label{eq:der}
\nabla f(x)=g'(\langle a,x\rangle)\cdot a.
\end{equation}
One way, how to use \eqref{eq:der} is to approximate the gradient of $f$ at a point with non-vanishing $g'(\langle a,x\rangle)$. By \eqref{eq:der}, it is then co-linear with $a$.
Once $a$ is recovered, one can use any one-dimensional sampling method to approximate $g$.

Another way to approximate $a$ is inspired by the technique of \emph{compressed sensing} \cite{CT, D}.
Taking directional derivatives of $f$ at $x$ results into
\[
\frac{\partial f(x)}{\partial \varphi}=\langle \nabla f(x),\varphi \rangle=g'(\langle a,x\rangle)\langle a,\varphi\rangle,
\]
i.e. it gives an access to the scalar product of $a$ with a chosen vector $\varphi$. If we assume, that most of the coordinates of $a$ are zero (or at least very small)
and choose the directions $\varphi_1,\dots,\varphi_m$ at random, one can recover $a$ effectively by the algorithms of sparse recovery.

Our aim is to fill some gaps left so far in the analysis done in \cite{paper3}. Although the possibility of 
extending the analysis also to functions defined on other domains than the unit ball was mentioned already in \cite{paper3}, no steps in this direction were done there.
We study in detail ridge functions
defined on the unit cube $[-1,1]^d$. The crucial component of our analysis is the use of the sign of a vector $\sgn (x)$,
which is defined componentwise. Although the mapping $x\to \sgn(x)$ is obviously not continuous, the mapping (for $a\in \R^d$ fixed)
\[
x\to \langle a, \sgn(x) \rangle
\]
is continuous at $a$ (and takes the value $\|a\|_{l_1^d}$ there). This observation allows to imitate the approach of \cite{paper3}
and to adapt it to this setting. Let us remark, that all our approximation schemes recover first an approximation of the vector $a\in\R^d$. Afterwards,
the problem becomes essentially one-dimensional and a good approximation of $f$ by a limited number of sampling points can then be recovered by
many classical methods, i.e. by spline approximation. We will therefore concentrate on an effective recovery of an approximation of $a$ and the approximation of $f$ will be given only implicitly.

Another topic only briefly discussed in \cite{paper3} was the recovery of ridge functions from noisy measurements, which is an important step for every
possible application of the methods so far. Furthermore, our analysis as well as the approach of \cite{paper3} or even the classical results of \cite{BP}
are based on approximation of first (or higher) order derivatives by differences, which poses naturally the question on numerical stability of the presented algorithms.
We present an algorithm based on the Dantzig selector of \cite{paper6}, which allows for recovery of a ridge function also in this setting. It turns out, that in the case of a small
step size $h>0$, the first order differences can not be evaluated with high enough precision. On the other hand, for a large step size $h$ the first order differences do not approximate
the first order derivatives well enough. Typically, there is therefore an $h>0$, for which an optimal degree of approximation is achieved.

Next thing we discuss is the robustness of the methods developed. We show that (without much additional effort) it can be applied also for uniform recovery of translated radial functions
$f(x)=g(\|a-x\|^2_{l_2^d})$, which are constant along co-centered spheres instead of parallel hyperplanes. Similarly to the model of ridge functions, both the center $a\in\R^d$ and the
univariate function $g$ are unknown.

Finally, we close the paper with few numerical simulations of the algorithms presented. They highlight the surprising fact, that their accuracy
\emph{improves} with increasing dimension. This is essentially based on the use of \emph{concentration of measure} phenomenon in the underlying theory and
goes in line with similar observations made in the area of compressed sensing.

The paper is structured as follows. Section 2 collects some necessary notation and certain basic facts on sparse recovery from the area of \emph{compressed sensing}.
Section 3 extends the analysis of \cite{paper3} to the setting of ridge functions defined on the unit cube. Section 4 treats the recovery of ridge functions defined on the unit ball
from noisy measurements. Section 5 studies the translated radial functions $f(x)=g(\|a-x\|_{l_2^d}^2)$ and Section 6 closes with numerical examples.

\section{Preliminaries}

In this section we collect some notation and give an overview of results from the area of compressed sensing, which we shall need later on.

\subsection{Notation}

For a given vector $x\in\R^d$ and $0\le p\le \infty$ we define
\begin{align*}
    \|x\|_{l_p^d}&:=\begin{cases}\Bigl(\sum\limits_{i=1}^d\vert x_i\vert^p\Bigr)^{\frac1p} & \text{if}\ 0<p<\infty,\\
    \#\{i\mid x_i\neq0\} & \text{if}\ p=0,\\    
    \max\limits_{i=i,\ldots,d}\vert x_i\vert & \text{if}\ p=\infty,\\
    \end{cases}
\end{align*}
where $\#A$ denotes the cardinality of the set $A$. 

This notation is further complemented by putting for $0<p<\infty$
\[
        \|x\|_{l_{p,\infty}^d}:=\max\limits_{k=1,\ldots,d}k^{\frac1p} x_{(k)},
\]
where $x_{(k)}$, $k=1,\ldots,d$ denotes the non-increasing rearrangement of the absolute entries of $x$, i.e.
$x_{(1)}\geq x_{(2)}\geq\ldots\geq x_{(d)}\geq0$ and $x_{(j)}=|x_{\sigma(j)}|$ for some permutation $\sigma:\{1,\dots,d\}\to\{1,\dots,d\}$ and all $j=1,\dots,d.$

It is a very well known fact, that $\|\cdot\|_{\ell_p^d}$ is a norm for $1\le p \le \infty$ and a quasi-norm if $0<p\le 1.$ Also $\|\cdot\|_{\ell_{p,\infty}^d}$ is a quasi-norm for every $0<p<\infty$.
If $p=2$, the space $\ell_2^d$
is a Hilbert space with the usual inner product given by
\begin{align*}
        \langle x,y\rangle=x^Ty=\sum\limits_{i=1}^dx_iy_i,~x,y\in\R^d.
\end{align*}
 
If $1\le s \le d$ is a natural number, then a vector $x\in\R^d$ is called \emph{$s$-sparse} if it contains at most $s$ nonzero entries, i.e. $\|x\|_{l_0^d}\leq s$. The set of all $s$-sparse vectors is denoted by
\begin{align*}
\Sigma_s^d:=\{x\in\R^d\mid\|x\|_{l_0^d}\leq s\}.
\end{align*}
Finally, the best $s$-term approximation of a vector $x$ describes, how well can $x$ be approximated by $s$-sparse vectors.
\begin{definition}
	The \emph{best $s$-term approximation} of a given vector $x\in\R^d$ with respect to the $l_1^d$-norm is given by
	\begin{align*}
		\sigma_s(x)_{1}:=\min\limits_{z\in\Sigma^d_s}\|x-z\|_{l_1^d}.
	\end{align*}
\end{definition}

\subsection{Results from compressed sensing}

Next we recall some basic concepts and results from compressed sensing which we will use later. Compressed sensing emerged
in \cite{CT, CRT, D} as a method of recovery of sparse vectors $x$ from a small set of linear measurements $y=\Phi x.$
Since then, a vast literature on the subject appeared, concentrating on various aspects of the theory, and its applications.
As it is not our aim to develop the theory of compressed sensing, but rather to use it in approximation theory, we shall restrict
ourselves to the most important facts needed later on. We refer to \cite{BCKV, DDEK, FR2, paper7} for recent overviews of the field and more references.

We focus on the recovery of vectors from noisy measurements, i.e. we want to recover the vectors $x\in\R^d$ from $m<d$ linear measurements of the form
\begin{align}\label{compressed-sensing-setting}
	y=\Phi x+e+z,
\end{align}
where $\Phi\in\R^{m\times d}$ is the measurement matrix and the noise is a composition of two factors, namely of the deterministic noise $e\in\R^m$ and the random noise $z\in\R^m$.
Typically, we will assume, that $e$ is small (with respect to some $\ell_p^m$ norm) and that the components of $z$ are generated
independently according to a Gaussian distribution with small variance.

Obviously, some conditions have to be posed on $\Phi$, so that the recovery of $x$ from the measurements $y$ given by \eqref{compressed-sensing-setting} is possible.
The most usual one in the theory of compressed sensing is that the matrix $\Phi$ satisfies the \emph{restricted isometry property}. 

\begin{definition}
	The matrix $\Phi\in\R^{m\times d}$ satisfies the \emph{restricted isometry property} (RIP) of order $s\le d$ if there exists a constant $0<\delta<1$
	such that
	\begin{align*}
		(1-\delta)\|x\|^2_{l_2^d}\leq\|\Phi x\|_{l_2^m}^2\leq(1+\delta)\|x\|^2_{l_2^d}
	\end{align*}
	holds for all $s$-sparse vectors $x\in\Sigma_s^d$. The smallest constant $\delta$ for which this inequality holds is called the
	restricted isometry constant and we will denote it by $\delta_s$.
\end{definition}

In general it is very hard to show that a given matrix satisfies this RIP or not.
This is in particular the main reason why we will use random matrices, since it turns out that those matrices satisfy the RIP with overwhelming high probability.
We present a version of such a statement, which comes from \cite{paper2}.
\begin{theo}\label{rip}
	For every $0<\delta<1$ there exist constants $C_1,C_2>0$ depending on $\delta$ such that the random matrix $\Phi\in\R^{m\times d}$ with entries generated independently as
	\begin{align}\label{eq:Bern}
		\ph_{ij}=\frac{1}{\sqrt{m}}\begin{cases}+1~\text{with probability}~1/2,\\-1~\text{with probability}~1/2\end{cases}
	\end{align}
	satisfies the RIP of order $s$ for each $s\leq(C_2 m)/\log(d/m)$ with RIP constant $\delta_{s}\leq\delta$ with probability at least
	\begin{align*}
		1-2e^{-C_1 m}.
	\end{align*}
\end{theo}
A matrix $\Phi$ generated by \eqref{eq:Bern} is called \emph{normalized Bernoulli matrix}. 
For the sake of simplicity, we work with Bernoulli sensing matrices, but note
that most of the statements presented below remain true for other classes of random matrices, c.f. \cite[Section 5]{paper5}.

Next we present several recovery results for our starting problem \eqref{compressed-sensing-setting}.
The first result of this kind deals with the case of  exact measurements (i.e. $e=z=0$) and uses the so called $l_1^d$-minimizer, cf. \cite[Theorem 4.3]{paper4}.
\begin{theo}\label{RIP-reconstruction}
Let $\Phi\in\R^{m\times d}$ satisfy the RIP of order $2s$ with constant $\delta_{2s}\leq\delta<1/3$.
Let $x\in\R^d$ and let us denote $y=\Phi x$. Finally, let $\Delta_{l_1^d}(y)\in\R^d$ be the solution of the minimization problem
\begin{align}\label{eq:l1}
\min\limits_{w\in\R^d}\|w\|_{l_1^d}\quad \text{subject to}\quad \Phi w=y.
\end{align}
Then it holds
\begin{align*}
\|x-\Delta_{l_1^d}(y)\|_{l_1^d}=\|x-\Delta_{l_1^d}(\Phi x)\|_{l_1^d}\leq C_0 \sigma_s^d(x)_{1}
\end{align*}
with constant $C_0$ depending only on $\delta$. 
\end{theo}

This theorem implies that $s$-sparse vectors are recovered exactly by the $l_1^d$-minimizer \eqref{eq:l1} in the noise-free setting, since $\sigma_s^d(x)_{1}=0$ holds for every $x\in\Sigma_s^d$.
To deal with the deterministic noise $e$, we shall need some more information about the geometrical properties of Bernoulli matrices. In particular, we will make use of Theorem 3.5 and Theorem 4.1 of \cite{paper5},
cf. also \cite{LPRTJ}.

\begin{theo}\label{surjective-deterministic-noise}
Let $\Phi\in\R^{m\times d}$ be a normalized Bernoulli matrix and let $d\geq (\log 6)^2m$.
Let $U_J= \{y\in\R^m: \|y\|_J\le 1\}$, where
\[
\|y\|_J=\max\left\{\sqrt{m}\|y\|_{l_\infty^m}; \sqrt{\frac{m}{\log(d/m)}}\|y\|_{l_2^m}\right\}.
\]
\begin{enumerate} \item[(i)] Then there exists an absolute constant $C_3>0$ such that with probability at least $1-e^{-\sqrt{dm}}$ for every $y\in U_J$ there is an $x\in\R^d$, such that $\Phi x=y$
and $\|x\|_{l_1^d}\le C_3$.
\item[(ii)] Let $\delta>0$ and let $C_1$ and $C_2$ be the constants from Theorem \ref{rip}. 
Then there exists an absolute constant $C_3$ and a constant $C_4$ depending on $\delta$ such that, with probability at least $1-2e^{-C_1 m}-e^{-\sqrt{md}}$, for each $y\in U_J$ 
there exists a vector $x\in\R^d$ with $\Phi x=y$, $\|x\|_{l_1^d}\leq C_3$ and $\|x\|_{l_2^d}\leq C_4\sqrt{\log(d/m)/m}$. 
\end{enumerate}
\end{theo}

We will use those two theorems to handle the deterministic noise $e$.
Further we need a similar result to handle the random noise $z$, therefore we recall the \emph{Dantzig selector} from \cite{paper6}.

\begin{definition}[Dantzig selector]
	For a matrix $\Phi\in\R^{m\times d}$ and constants $\lambda_d,\sigma>0$ the \emph{Dantzig selector} $\Delta_{DS}(y)\in\R^d$ of an input vector $y\in\R^m$
	is defined as the solution of the minimization problem
	\begin{align}\label{eq:Dantz}
		\min\limits_{w\in\R^d} \|w\|_{l_1^d}\quad \text{subject to}\quad \|\Phi^T (y-\Phi w)\|_{l_\infty^d}\leq\lambda_d\sigma.
	\end{align}
\end{definition}

\begin{remark}\label{rem:setting}
In what follows we shall use several parameters as the description of the typical frame of compressed sensing.
First, we take $m\le d$ to be natural numbers and denote by $\Phi\in\R^{m\times d}$ the normalized Bernoulli matrix \eqref{eq:Bern}.
We put $\delta:=1/6$ and denote by $C_1$ and $C_2$ the constants appearing in Theorem \ref{rip}.
Next, we assume that the natural numbers $s\le m\le d$ satisfy
\begin{equation}\label{eq:setting}
d\ge (\log 6)^2 m\quad\text{and}\quad 3s\leq(C_2 m)/\log(d/m).
\end{equation}
Hence, by Theorem \ref{rip}, $\Phi$ has (with high probability) the RIP of order $3s$ with a constant at most 1/6.
\end{remark}

Now we can use Theorem 1.3 of \cite{paper6} to handle the random noise $z$.

\begin{theo}\label{dantzig}
Let $s,m,d$ be natural numbers with \eqref{eq:setting} and let $\Phi\in\R^{m\times d}$ be a normalized Bernoulli matrix.
Let
	\begin{align*}
		y=\Phi x + z
	\end{align*}
	for $x\in\R^d$ with $\|x\|_{p,\infty}\leq R$, $0< p\leq 1$, and $z\in\R^m$ with independent entries $z_i\sim\mathcal{N}(0,\sigma^2)$.
	Then there exists a constant $C_5$ such that the Dantzig Selector 
	(with $\lambda_d=\sqrt{2\log d}$) satisfies
	\begin{align*}
		\|\Delta_{DS}(y)-x\|_{l_2^d}^2\leq\min\limits_{1\leq s_*\leq s}2C_5\log d\left(s_*\sigma^2+R^2s_*^{-2(1/p-1/2)}\right)
	\end{align*}	
	with high probability.
\end{theo}

Combining Theorem \ref{surjective-deterministic-noise} and Theorem \ref{dantzig} we get the following new result.

\begin{theo}\label{random+deterministic}
Let $s,m,d$ be natural numbers with \eqref{eq:setting} and let $\Phi\in\R^{m\times d}$ be a normalized Bernoulli matrix.
	For $x\in\R^d$ and $e,z\in\R^m$ with $\|x\|_{l_{1,\infty}^d}\leq R$, $\|e\|_{l_2^d}\leq c\,\eps\sqrt{\log(d/m)}$, $\|e\|_{l_\infty^d}\leq c\,\eps$ 
	and $z_i\sim\mathcal{N}(0,\sigma^2)$ for some constants $R,\sigma,\eps,c>0$ let
	\begin{align*}
		y=\Phi x+e+z.
	\end{align*}
	Then there exist constants $C_5,C_6,C_7$ such that the Dantzig selector $\Delta_{DS}$ (with $\lambda_d=\sqrt{2\log d}$)
	applied to $y$ satisfies the estimate
	\begin{align*}
		\|\Delta_{DS}(y)-x\|_{l_2^d}\leq \left(\min\limits_{1\leq s_*\leq s}2 C_5 \log d \left(s_*\sigma^2+\tilde R^2 s_*^{-1}\right)\right)^{\frac 12} + C_7\frac{\eps\sqrt{m}}{\sqrt{s}}
	\end{align*}
	with high probability, where $\tilde R= 2(R+2C_6\eps\sqrt{m})$.
\end{theo}

\begin{proof}
It follows from the assumptions that $\|e\|_J\le c\,\varepsilon\sqrt{m}$. Then we use Theorem \ref{surjective-deterministic-noise} (ii) to find a vector $u\in\R^d$, such that 
\begin{align*}
\Phi u&=e,\\
\|u\|_{l_1^d}&\le C_3 \|e\|_J\le C_3\,c\,\varepsilon\sqrt{m},\\
\|u\|_{l_2^d}&\le C_4\sqrt{\log(d/m)/m}\|e\|_J\le C_4\,c\,\varepsilon\sqrt{\log(d/m)}.
\end{align*}

Further we apply the triangle inequality for the $\|\cdot\|_{1,\infty}$ quasinorm (see, for instance, Lemma 2.7 in \cite{paper7}) to get
\begin{align*}
	\|x+u\|_{l_{1,\infty}^d}&\leq 2\bigl(\|x\|_{l_{1,\infty}^d}+\|u\|_{l_{1,\infty}^d}\bigr)\leq 2\bigl(\|x\|_{l_{1,\infty}^d}+\|u\|_{l_{1}^d}\bigr)\\
&\leq 2(R+C_6\eps\sqrt{m})=:\tilde R.
\end{align*}
Finally, applying Theorem \ref{dantzig} (with $p=1$) we get
\begin{align*}
	\|\Delta_{DS}(y)-x\|_{l_2^d}
	&=\|\Delta_{DS}(\Phi x + e +z)-x\|_{l_2^d}\\
	&\leq \|\Delta_{DS}(\Phi(x+u)+z)-(x+u)\|_{l_2^d}+\|u\|_{l_2^d}\\
	&\leq \left(\min\limits_{1\leq s_*\leq s}2C_5\log d\left(s_*\sigma^2+\tilde R^2s_*^{-1}\right)\right)^{\frac12}+C_7\frac{\eps\sqrt{m}}{\sqrt{s}}
\end{align*}
	which finishes the proof.
\end{proof}

\section{Approximation of ridge functions defined on cubes}\label{ridge-on-cubes}

In this section we consider uniform approximation of ridge functions of the form
\begin{align}\label{eq:ridge}
f\colon [-1,1]^d\to\R,~x\mapsto g(\langle a,x\rangle).
\end{align}
We assume that both the \emph{ridge vector} $a\in\R^d$ 
and the univariate function $g$ (also called \emph{ridge profile}) are unknown.

First, we note that the problem is invariant with respect to scaling.
Suppose that $f$ is a ridge function with representation $f(x)=g(\langle a,x\rangle)$.
Then for any scalar $\lambda\in\R\backslash\{0\}$ we put $\tilde g(x):= g(\tfrac{1}{\lambda}x)$ and $\tilde a:=\lambda a$ to get
another representation of $f$ in the form of \eqref{eq:ridge}, namely
\begin{align*}
	\tilde g(\langle\tilde a,x\rangle)=\tilde g(\langle\lambda a,x\rangle)
	=g\left(\langle\frac{1}{\lambda}\lambda a,x\rangle\right)=g(\langle a,x\rangle)=f(x).
\end{align*}
Thus we can pose a scaling condition on $a$ without any loss of generality. Furthermore, if $g'(0)\not=0$, we can switch
from $a$ to $-a$, and obtain a ridge representation of $f$ with $g'(0)>0.$

In \cite{paper3}, the scaling condition $\|a\|_{l_2^d}=1$ was assumed. This fitted together with both the scalar product structure
used in the definition of $f$, as well as with the geometry of the domain of $f$ used in \cite{paper3}, namely the Euclidean unit ball.

It is easy to observe, that it will be more convenient for us to work with the $\ell_1^d$-norm of $a$.
Indeed, let us consider that the ridge profile $g(t)=t$ is known, i.e. that we have  $f(x)=\langle a,x\rangle$ for some (unknown) $a\in\R^d$,
and let us assume, that we have an $l_1^d$-approximation $\hat a$ of $a$ with $\|a-\hat a\|_{l_1^d}\le \eps$. Then H\"older's inequality gives us
\begin{align*}
	\|\hat f-f\|_{\infty}
	=\sup_{x\in [-1,1]^d}\lvert \langle a-\hat a,x\rangle\rvert
	\leq\sup_{x\in [-1,1]^d}\|a-\hat a\|_{l_1^d}\|x\|_{l^d_\infty}
	\le \eps.
\end{align*}

In what follows we shall therefore assume that
\begin{equation}\label{eq:assum1}
\|a\|_{l_1^d}=1
\end{equation}
and that $g$ is a univariate function defined on $I=[-1,1].$
We further assume that $g$ and $g'$ are Lipschitz continuous with constants $c_0,c_1>0$, i.e.
\begin{align}
\label{eq:assum2}\vert g(t_1)-g(t_2)\vert&\leq c_0\vert t_1-t_2\vert,\\
\label{eq:assum3}\vert g'(t_1)-g'(t_2)\vert&\leq c_1\vert t_1-t_2\vert
\end{align}
holds for all $t_1,t_2\in I=[-1,1]$.
Finally, we assume that
\begin{equation}\label{eq:assum4}
g'(0)>0
\end{equation}
as it is known, cf. \cite{MUV}, that approximation of ridge functions may be intractable if this condition is left out.

\subsection{Approximation scheme without sparsity}
In this part we evolve an approximation scheme for ridge functions with an arbitrary ridge vector $a\in\R^d$,
merely assuming the right normalization \eqref{eq:assum1}.
After this we consider the same problem with an additional sparsity condition on $a$, where we will use results from compressed sensing to reduce
the number of samples.

Motivated by the formula \eqref{eq:der} for $x=0$
\begin{equation}\label{eq:nabla}
\nabla f(0)=g'(0)a,
\end{equation}
we set for a small constant $h>0$ and $i\in\{1,\ldots,d\}$
\begin{align}\label{eq:nonoise}
\tilde a_i:=\frac{f(h e_i)-f(0)}{h},
\end{align}
where $e_1,\dots,e_d$ are the usual canonical basis vectors of $\R^d.$
As expected, it turns out that $\tilde a_i$ is a good approximation of $g'(0)a_i$ as the mean value theorem gives
\begin{align*}
	\tilde a_i
	&=\frac{f(h e_i)-f(0)}{h}
	=\frac{g(h\langle a,e_i\rangle)-g(0)}{h}
	=g'(\xi_{h,i}) a_i
\end{align*}
for some $\xi_{h,i}\in(-\vert ha_i\vert,\vert ha_i\vert)$. And for the $\ell_1^d$-approximation we obtain
\begin{align}
\notag\|\tilde a-g'(0)a\|_{l_1^d}
&=\sum\limits_{i=1}^d\lvert\tilde a_i-g'(0)a_i\rvert
=\sum\limits_{i=1}^d\lvert g'(\xi_{h,i})-g'(0)\rvert \lvert a_i\rvert\\
\label{eq:g1}&\leq \sum\limits_{i=1}^d c_1\lvert \xi_{h,i}\rvert \lvert a_i\rvert
\leq \sum\limits_{i=1}^d c_1\lvert ha_i\rvert \lvert a_i\rvert
=c_1 h\sum\limits_{i=1}^d\lvert a_i\rvert^2\\
\notag&\leq c_1 h.
\end{align}
Thus $\tilde a$ is a good approximation to $g'(0) a$ and since we want an approximation to $a$ and we know that $a$ is $l_1^d$-normalized we set
\begin{align*}
	\hat a:=\frac{\tilde a}{\|\tilde a\|_{l_1^d}}.
\end{align*}
Now we have to estimate the difference between $a$ and $\hat a$, therefore we will use a variant of Lemma 3.4 of \cite{paper3}.
\begin{lemma}\label{stability-subspaces}
	Let $x\in\R^d$ with $\|x\|_{l_1^d}=1$, $\tilde x\in\R^d\backslash\{0\}$ and $\lambda\in\R$.
	Then it holds
	\begin{align*}
		\left\|\sgn(\lambda)\frac{\tilde x}{\|\tilde x\|_{l_1^d}}-x\right\|_{l_1^d}
		\leq\frac{2\|\tilde x-\lambda x\|_{l_1^d}}{\|\tilde x\|_{l_1^d}}.
	\end{align*}
\end{lemma}
\begin{proof}
	This lemma is a direct consequence of the triangle inequality. First we obtain
\begin{align*}
	\left\vert\|\tilde x\|_{l_1^d}-\vert\lambda\vert\right\vert
	=\left\vert\|\tilde x\|_{l_1^d}-\|\lambda x\|_{l_1^d}\right\vert
	\leq\|\tilde x-\lambda x\|_{l_1^d}
\end{align*}
and therefore
\begin{align*}
	\left\|\sgn(\lambda)\frac{\tilde x}{\|\tilde x\|_{l_1^d}}-x\right\|_{l_1^d}
	&\leq\left\|\frac{\sgn(\lambda)\tilde x-\sgn(\lambda)\lambda x}{\|\tilde x\|_{l_1^d}}\right\|_{l_1^d}
		+\left\|\frac{\sgn(\lambda)\lambda x-x\|\tilde x\|_{l_1^d}}{\|\tilde x\|_{l_1^d}}\right\|_{l_1^d}\\
	&=\frac{\left\|\tilde x-\lambda x\right\|_{l_1^d}}{\|\tilde x\|_{l_1^d}}
		+\left\|\frac{(\vert\lambda\vert-\|\tilde x\|_{l_1^d}) x}{\|\tilde x\|_{l_1^d}}\right\|_{l_1^d}
	\leq \frac{2\|\tilde x-\lambda x\|_{l_1^d}}{\|\tilde x\|_{l_1^d}},
\end{align*}
which proves the claim.
\end{proof}

\begin{remark}
We only used the triangle inequality to prove the previous lemma.
Thus the lemma remains true for any norm on $\R^d$.
\end{remark}
Applying this lemma to our case it holds with \eqref{eq:g1} and the assumption \eqref{eq:assum4}
\begin{align}\label{eq:stab1}
	\|\sgn(g'(0))\hat a-a\|_{l_1^d}=\|\hat a-a\|_{l_1^d}\leq \frac{2c_1h}{\|\tilde a\|_{l_1^d}}.
\end{align}

Although we now know that $\hat a$ is a good approximation of $a$, it is still not clear how to define the uniform approximation $\hat f$ of $f$.
The naive approach (used with success in \cite{paper3} for ridge functions defined on the Euclidean unit ball) is
to sample $f$ along $\hat a$, i.e. to put $\hat g(t):=f(t\hat a)$, and then define $\hat f(x):=\hat g(\langle\hat a,x \rangle)$.
But when trying to estimate $\|f-\hat f\|_\infty$, we would need to ensure that $\langle \hat a,a\rangle$ is close to 1.
This was indeed the case in \cite{paper3}, where an estimate on $\|\hat a-a\|_{\ell_2^d}$ was obtained, but it is not true any more in our setting
of $\ell_1^d$ approximation.

On the other hand, because of the normalization of $a$, we have
\begin{align*}
	\langle a,\sgn(a)\rangle =\sum\limits_{i=1}^d a_i\cdot\sgn(a_i)=\sum\limits_{i=1}^d \vert a_i\vert=\|a\|_{l_1^d}=1,
\end{align*}
where we defined the \emph{sign} of a vector $x\in\R^d$ entrywise, i.e.
\begin{align*}
	\sgn(x):=(\sgn(x_i))_i\in\R^d.
\end{align*}
Note that this function is discontinuous, hence $\sgn(a)$ and $\sgn(\hat a)$ can be far from each other,
even if the difference $\|a-\hat a\|_{l_1^d}$ is small.
Nevertheless their scalar product with $a$ is nearly the same as  H\"older's inequality gives
\begin{align}
\notag\vert\langle a,\sgn(a)-\sgn(\hat a)\rangle\vert
&=\vert\langle a,\sgn(a)\rangle-\langle \hat a,\sgn(\hat a)\rangle -\langle a-\hat a,\sgn(\hat a)\rangle\vert\\
\label{eq:a1}&\leq\|a-\hat a\|_{l_1^d}\|\sgn(\hat a)\|_{l_\infty^d}=\|a-\hat a\|_{l_1^d}.
\end{align}
Thus we define
\begin{align}\label{def:g}
\hat g\colon[-1,1]\to\R,~t\mapsto f\big(t\cdot\sgn(\hat a)\big)
\end{align}
and
\begin{align}\label{def:f}
\hat f(x)=\hat g(\langle\hat a,x\rangle).
\end{align}
Let us summarize our approximation algorithm as follows.
\vskip.2cm\noindent
\framebox[\textwidth][s]
{
\begin{minipage}{0.9\textwidth}
\vskip.2cm\centerline{{\bf Algorithm A}}\vskip-.10cm
\noindent\makebox[\linewidth]{\rule{10cm}{0.6pt}}
\begin{itemize}[leftmargin=*]
\item \emph{Input:} Ridge function $f(x)=g(\langle a,x\rangle)$ with \eqref{eq:assum1}-\eqref{eq:assum4} and $h>0$ small
\item Put $\displaystyle \tilde a_i:=\frac{f(h e_i)-f(0)}{h}, i=1,\dots,d$
\item Put $\displaystyle \hat a:=\frac{\tilde a}{\|\tilde a\|_{l_1^d}}$
\item Put $\hat g(t)=f(t\cdot \sgn(\hat a))$ and $\hat f(x)=\hat g(\langle\hat a,x\rangle)$
\item \emph{Output:} $\hat f$
\end{itemize}
\vskip.2cm
\end{minipage}
}
\vskip.2cm

We formulate the approximation properties of Algorithm A as the following theorem.

\begin{theo}\label{thm:AlgA}
Let $f\colon[-1,1]^d\to\R$ be a ridge function with $f(x)=g(\langle a,x\rangle)$ for some $a\in\R^d$ 
with \eqref{eq:assum1} and a differentiable function $g\colon[-1,1]\to \R$ with \eqref{eq:assum2}-\eqref{eq:assum4}. For $h>0$ we construct the
function $\hat f$ as described in Algorithm A. Then
\begin{align}\label{eq:unest1}
\|f-\hat f\|_{\infty}\leq 2c_0\|\hat a-a\|_{l_1^d}\leq\frac{4c_0c_1h}{g'(0)-c_1h},
\end{align}
where the last inequality only holds if $g'(0)-c_1h$ is positive.
\end{theo}
\begin{proof}
First, we show that
\begin{align}\label{eq:unest2}
	\|\hat a-a\|_{l_1^d}
	\leq\frac{2hc_1}{\|\tilde a\|_{l_1^d}}
	\leq\frac{2hc_1}{g'(0)-c_1h},
\end{align}
where the last inequality only holds if $g'(0)-c_1h$ is positive.

Due to \eqref{eq:stab1}, we only have to show the last inequality of \eqref{eq:unest2}. Since $g'$ is Lipschitz continuous with Lipschitz constant $c_1$ we have
for any $y\in[-h,h]$
\begin{align*}
	g'(0)-\vert g'(y)\vert \leq\vert g'(0)-g'(y)\vert \leq c_1\vert 0-y\vert \leq c_1 h
\end{align*}
and therefore
\begin{align*}
|g'(y)|\geq g'(0)-c_1 h.
\end{align*}
With $\tilde a_i=g'(\xi_{h,i})a_i$ for some $\xi_{h,i}\in(-\vert h a_i\vert,\vert h a_i\vert)\subset[-h,h]$, it follows by the triangle inequality and \eqref{eq:g1}
\begin{align*}
\|\tilde a\|_{l_1^d}\ge \|g'(0)a\|_{l_1^d} - \|\tilde a - g'(0)a\|_{l_1^d}
\ge g'(0)-c_1 h
\end{align*}
which proves \eqref{eq:unest2} and the second inequality in \eqref{eq:unest1}.

To prove the first inequality in \eqref{eq:unest1}, we use \eqref{eq:assum2} and \eqref{eq:a1}
to show that $\hat g$ is a good uniform approximation of $g$ on $[-1,1]$. We obtain
\begin{align}
\notag\vert g(t)-\hat g(t)\vert&=\vert g(t)-g(\langle a,t\cdot\sgn(\hat a)\rangle)\vert\leq c_0\vert t-t\langle a,\sgn(\hat a)\rangle\vert\\
\label{eq:g_est}&= c_0\vert t\vert\left\vert\langle a,\sgn(a)-\sgn(\hat a)\rangle\right\vert\leq c_0\|a-\hat a\|_{l_1^d}
\end{align}
for each $t\in[-1,1]$. Finally, we combine this estimate with the definition of $\hat f$ as given in \eqref{def:f} and arrive at
\begin{align}
\notag\vert \hat f(x)-f(x)\vert&=\vert \hat g(\langle\hat a,x\rangle)-g(\langle a,x\rangle)\vert\\
\label{eq:f_est}&\leq\vert \hat g(\langle\hat a,x\rangle)-g(\langle\hat a,x\rangle)\vert+\vert g(\langle\hat a,x\rangle)-g(\langle a,x\rangle)\vert\\
\notag&\leq c_0\|a-\hat a\|_{l_1^d}+c_0\vert\langle a-\hat a,x\rangle\vert\leq 2c_0\|a-\hat a\|_{l_1^d}.
\end{align}
\end{proof}

\begin{remark} \label{rem:3.4}
\begin{enumerate}
\item [(i)] The estimate \eqref{eq:unest1} depends heavily on the value of $g'(0)$. Especially, the approximation becomes difficult, when
this value gets smaller and \eqref{eq:unest1} becomes void if $g'(0)=0.$ This is a very well known aspect of approximation of ridge functions,
which was studied in a great detail in \cite{MUV}. We refer also to a slightly weaker condition used in \cite{paper3}.
\item [(ii)] If $\|a\|_{l_2^d}$ is small, the following improvement of \eqref{eq:unest1} becomes of interest. First, we observe that \eqref{eq:g1}
can be improved to $\|\tilde a-g'(0)a\|_{l_1^d}\le c_1h\|a\|^2_{l_2^d}$, which results into
\[
\|\hat a-a\|_{l_1^d}\le \frac{2c_1h\|a\|^2_{l_2^d}}{\|\tilde a\|_{l_1^d}}.
\]
Finally, this allows to improve \eqref{eq:unest1} to
\[
\|f-\hat f\|_{\infty}\leq\frac{4c_0c_1h\|a\|^2_{l_2^d}}{g'(0)-c_1h\|a\|^2_{l_2^d}}.
\]
\end{enumerate}
\end{remark}

\subsection{Approximation with sparsity}
In this subsection we assume that the ridge vector $a\in\R^d$ is not only $\ell_1^d$-normalized, but satisfies also some sparsity condition, i.e. most of the entries of $a$ are zero
or at least very small. We will use techniques of compressed sensing
to address the approximation of the ridge vector $a$, afterwards we obtain an approximation of $f$ in the same way as before.

Let $\Phi\in\R^{m\times d}$ be a normalized Bernoulli matrix and let $\varphi_1,\dots,\varphi_m$ be its rows. Taking their scalar product with the quantities in \eqref{eq:nabla},
we obtain
\begin{equation}\label{eq:diff}
\frac{\partial f}{\partial \varphi_j}(0)=\langle \nabla f(0),\varphi_j\rangle = g'(0)\langle a,\varphi_j\rangle.
\end{equation}

We use again first order differences as an approximation of the directional derivatives in \eqref{eq:diff}, i.e. we set
\begin{align*}
\tilde b_j:=\frac{f(h\ph_j)-f(0)}{h}.
\end{align*}
As in the previous section the mean value theorem gives the existence of some
$\xi_{h,j}$ with $|\xi_{h,j}|\le |h|\cdot \vert\langle a,\ph_j\rangle\vert$ such that
\begin{align*}
\tilde b_j=g'(\xi_{h,j})\langle a,\ph_j\rangle.
\end{align*}
In this sense, we expect $\tilde b_j$ to be a good approximation of $g'(0)\langle a,\varphi_j \rangle$ and $\tilde b$ to be 
a good approximation of $g'(0)\Phi a.$ Hence, we recover $\tilde a$ through $\ell_1$-minimization. From this point on, we may continue as before.
Let us summarize this procedure as the Algorithm B.
\vskip.2cm\noindent
\framebox[\textwidth][s]
{
\begin{minipage}{0.9\textwidth}
\vskip.2cm\centerline{{\bf Algorithm B}}\vskip-.10cm
\noindent\makebox[\linewidth]{\rule{10cm}{0.6pt}}
\begin{itemize}[leftmargin=*]
\item \emph{Input:} Ridge function $f(x)=g(\langle a,x\rangle)$ with \eqref{eq:assum1}-\eqref{eq:assum4}, $h>0$ small and $m\le d/(\log 6)^2$
\item Take $\Phi\in\R^{m\times d}$ a normalized Bernoulli matrix, cf. \eqref{eq:Bern}
\item Put $\displaystyle \tilde b_j:=\frac{f(h\ph_j)-f(0)}{h},~j=1,\ldots,m$
\item Put $\displaystyle\tilde a:=\Delta_{l_1^d}(\tilde b)=\argmin_{w\in\R^d}\ \|w\|_{l_1^d}\ \text{s.t.}\ \Phi w=\tilde b$
\item Put $\displaystyle \hat a:=\frac{\tilde a}{\|\tilde a\|_{l_1^d}}$
\item Put $\hat g(t)=f(t\cdot \sgn(\hat a))$ and $\hat f(x)=\hat g(\langle\hat a,x\rangle)$
\item \emph{Output:} $\hat f$
\end{itemize}
\vskip.2cm
\end{minipage}
}
\vskip.2cm

\begin{theo}\label{thm:AlgB}
Let $f\colon[-1,1]^d\to\R$ be a ridge function with $f(x)=g(\langle a,x\rangle)$ for some vector $a\in\R^d$ with \eqref{eq:assum1} and some differentiable
function $g\colon[-1,1]\to \R$ with \eqref{eq:assum2}--\eqref{eq:assum4}.
Let $d\geq(\log 6)^2m$ and $h>0$ be fixed. 
Then there exist some constants $C_0',C_1,C_2>0$,  such that for every positive integer $s$ with $2s\leq(C_2 m)/\log(d/m)$ the function $\hat f$ constructed in Algorithm B satisfies
\begin{align}\label{eq:funif}
\|f-\hat f\|_{\infty}\leq 2c_0\|\hat a -a\|_{l_1^d}\leq 2c_0\,\err(a,\hat a),
\end{align}
where
\begin{align*}
\err(a,\hat a):=C_0'\cdot\frac{g'(0)\cdot\sigma_s^d(a)_{1}+h}{g'(0)(1-\sigma_s^d(a)_1)-c_1h},
\end{align*}
with probability at least $1-e^{-\sqrt{md}}-e^{-C_1 m}$ provided the denominator in the expression for $\err(a,\hat a)$ is positive.
\end{theo}
\begin{proof}
The first inequality in \eqref{eq:funif} follows again by \eqref{eq:g_est} combined with \eqref{eq:f_est}.

Setting $\tilde b:=(\tilde b_1,\ldots,\tilde b_m)^T\in\R^m$ and $b:=g'(0)\Phi a\in\R^m$ we get
\begin{align*}
	\|\tilde b-b\|_{l_1^d}
	&=\sum\limits_{j=1}^m\vert g'(\xi_{h,j})\langle a,\ph_j\rangle-g'(0)\langle a,\ph_j\rangle\vert
	=\sum\limits_{j=1}^m\vert g'(\xi_{h,j})-g'(0)\vert \vert\langle a,\ph_j\rangle\vert\\
	&\leq\sum\limits_{j=1}^m c_1h\vert\langle a,\ph_j\rangle\vert^2
	\leq c_1 h \sum\limits_{j=1}^m \|a\|_{l_1^d}^2 \|\ph_j\|_{l_\infty^d}^2\\
	&=c_1 h.
\end{align*}
Therefore we obtain
\begin{align}
	\tilde b=b+e=g'(0)\Phi a+e
\end{align}
for $e\in\R^m$ with $\|e\|_{l_1^m}\leq c_1 h$ and, similarly, $\|e\|_{l_\infty^m}\leq c_1 h/m$ and $\|e\|_{l_2^m}\leq c_1h/\sqrt{m}$.
Hence, by using Theorem \ref{surjective-deterministic-noise} there exists some vector $u\in\R^d$ with $\Phi u=e$ and
\begin{align*}
	\|u\|_{l_1^d}
	&\leq\max\left\{\sqrt{m}\|e\|_{l_\infty^m}~;~\sqrt{\frac{m}{\log(d/m)}}\|e\|_{l_2^m}\right\}\\
	&\leq\max\left\{\frac{c_1h}{\sqrt{m}}~;~\frac{c_1h}{\sqrt{\log(d/m)}}\right\}\\
	&\leq\sqrt{2}c_1h,
\end{align*}
where we used $m\geq\log(d)$ and $d\geq(\log 6)^2m$ for the last inequality.
Take some $1/3>\delta>0$ fixed, e.g. $\delta=1/6$, and apply Theorem \ref{RIP-reconstruction} to $g'(0)a+u$. This gives us
\begin{align*}
	\|\Delta_{l_1^d}(\tilde b)-g'(0)a\|_{l_1^d}&=\|\Delta_{l_1^d}\left(\Phi(g'(0)a+u)\right)-g'(0)a\|_{l_1^d}\\
	&\leq \|\Delta_{l_1^d}\left(\Phi(g'(0)a+u)\right)-g'(0)a-u\|_{l_1^d} +\|u\|_{l_1^d}\\
	&\leq C_0\cdot\sigma_s^d\left(g'(0)a+u\right)_{1}+\|u\|_{l_1^d}\\
	&\leq C_0 g'(0)\cdot\sigma_s^d(a)_{1}+(1+C_0)\|u\|_{l_1^d}\\
	&\leq (1+C_0)\left(g'(0)\cdot\sigma_s^d(a)_{1}+\|u\|_{l_1^d}\right).
\end{align*}
Finally, by setting $\tilde a:=\Delta_{l_1^d}(\tilde b)$ and $\hat a:=\tilde a/\|\tilde a\|_{l_1^d}$,
Lemma \ref{stability-subspaces} provides
\begin{align*}
	\|a-\hat a\|_{l_1^d}
	&\leq2(1+C_0)\cdot\frac{g'(0)\cdot\sigma_s^d(a)_1+\|u\|_{l_1^d}}{\|\tilde a\|_{l_1^d}}\\
	&\leq 2\sqrt{2}(1+C_0)\cdot\frac{g'(0)\cdot\sigma_s^d(a)_1+c_1h}{\|\tilde a\|_{l_1^d}}.
\end{align*}
From this point on we can proceed as in the proof of Theorem \ref{thm:AlgA}. We can again estimate the $l_1^d$-norm of $\tilde a$ from below. We get
\begin{align*}
	\|\tilde a\|_{l_1^d}
	&=\|\Delta_{l_1^d}(\Phi(g'(0)a+u))\|_{l_1^d}\\
	&\geq \|g'(0)a+u\|_{l_1^d}-\|\Delta_{l_1^d}(\Phi(g'(0)a+u))-g'(0)a-u\|_{l_1^d}\\
	&\geq g'(0)\|a\|_{l_1^d}-\|u\|_{l_1^d} - \sigma_s^d(g'(0)a+u)_{1}\\
	&\geq g'(0)-2\|u\|_{l_1^d}-g'(0)\sigma_s^d(a)_{1}\\
	&\geq g'(0)-2\sqrt{2}c_1h-g'(0)\sigma_s^d(a)_{1}.
\end{align*}
Thus we can replace the norm $\|\tilde a\|_{l_1^d}$ with this expression (if it is positive) to get
\begin{align*}
	\|f-\hat f\|_{\infty}
	&\leq2c_0 C_0'\cdot\frac{ g'(0)\cdot\sigma_s^d(a)_{1}+c_1h}{\|\tilde a\|_{l_1^d}}\\
	&\leq C\cdot\frac{g'(0)\cdot\sigma_s^d(a)_{1}+h}{g'(0)(1-\sigma_s^d(a)_1)-c_1h}
\end{align*}
for some constant $C$ depending only on $c_0,c_1,C_0,C_0'$.
\end{proof}

\begin{remark}\label{rem:AlgB}~
\begin{enumerate}
\item[(i)] In particular, if $a$ is $s$-sparse, we get $\sigma_s^d(a)_1=0$ and, therefore,
\begin{align*}
	\|f-\hat f\|_\infty\leq C\frac{h}{g'(0)-c_1h}.
\end{align*}
\item[(ii)] The constant $C_0'$ can be chosen to be $C_0'=2\sqrt{2}(1+C_0)$ with $C_0$ being the constant from Theorem \ref{RIP-reconstruction}.
\item[(iii)] If the sparsity level of $a$ is $s\in\N$, the condition $2s\leq(C_2 m)/\log(d/m)$ implies
$m\geq 2s\log(d)/C_2$. Thus, in this case we need $m=O(s\log d)$ measurements to reconstruct the vector $a$.
\end{enumerate}
\end{remark}

\section{Approximation of ridge functions with noisy measurements}
\label{stability-ridge}

In this section we study another aspect of recovery of ridge functions, which was hardly discussed up to now in the literature.
We consider ridge functions defined on the unit ball as in \cite{paper3} but we assume that
the measurements are affected by random noise. In addition, we suppose that the vector $a$ satisfies a compressibility condition.

To be more precise, we consider ridge functions
\begin{align*}
f\colon B^d=\{x\in\R^d\mid\|x\|_{l_2^d}<1\}\to\R,~x\mapsto f(x)=g(\langle a,x\rangle).
\end{align*}
We assume, that the ridge vector $a\in\R^d$ is $l_2^d$-normalized 
\begin{equation}\label{eq:assump10}
\|a\|_{l_2^d}=1
\end{equation} 
and compressible in the following sense,
\begin{equation}\label{eq:assump11}
\|a\|_{l_1^d}\leq R,\quad  R>0.
\end{equation}
Furthermore, we assume that the ridge profile is a differentiable function $g\colon[-1,1]\to\R$ with \eqref{eq:assum2}--\eqref{eq:assum4}.

We shall use again the setting of Remark \ref{rem:setting}.
Let $d\geq (\log 6)^2m$ and let $\Phi\in\R^{m\times d}$ be a normalized Bernoulli matrix \eqref{eq:Bern} with rows $\ph_1,\ldots,\ph_m$. By
Theorem \ref{rip} it is ensured that $\Phi$ satisfies the RIP of order $2s$ with $0<\delta_{2s}\le \delta:=1/6$ with high probability
for every positive integer $s$ with $3s\leq(C_2 m)/\log(d/m)$, where the constant $C_2$ is the constant from Theorem \ref{rip}.


But in contrary to \eqref{eq:nonoise}, we now assume that the evaluation of $f$ is perturbed by noise.
To make the presentation technically simpler, we shall assume that the value $f(0)$ is given precisely (i.e. without noise).
This can be achieved (with high precision) by resampling the value $f(0)$ several times, and applying Hoeffding's inequality.

Hence, we set for $j=1,\ldots,m$ and a small constant $h>0$
\begin{align*}
	b_j:=\frac{(f(h\ph_j)+\tilde z_j)-f(0)}{h}
	=\frac{f(h\ph_j)-f(0)}{h}+\frac{\tilde z_j}{h}.
\end{align*}
We assume that the random noise $\tilde z=(\tilde z_1,\dots,\tilde z_m)^T\in\R^m$ has independent components $\tilde z_j\sim\mathcal{N}(0,\sigma^2)$.
Since $\tilde z_j$ are independent, it is well known that
\begin{align}\label{eq:noiseh}
z_j:=\frac{\tilde z_j}{h}\sim\mathcal{N}\left(0,\frac{\sigma^2}{h^2}\right)
\end{align}
are also independent. As in the case with exact measurements the mean value theorem gives us
\begin{align*}
	\frac{f(h\ph_j)-f(0)}{h}
	=\frac{g(\langle a,h\ph_j\rangle)-g(0)}{h}
	=g'(\xi_{h,j})\langle a,\ph_j\rangle
\end{align*}
for some real $\xi_{h,j}$ with $|\xi_{h,j}|\le \vert\langle a,h\ph_j\rangle\vert$, hence
\begin{align*}
	b_j=g'(\xi_{h,j})\langle a,\ph_j\rangle+z_j.
\end{align*}
To recover the vector $a$ from these measurements let us first define the deterministic noise $e\in\R^m$ by
\begin{equation}\label{eq:noise1}
e_j=\langle a,\ph_j\rangle(g'(\xi_{h,j})-g'(0)),\quad j=1,\dots,m,
\end{equation}
i.e.
\begin{align}\label{eq:defb}
	b=g'(0)\Phi a+e+z.
\end{align}

We then recover $\hat a$ with the help of Dantzig selector \eqref{eq:Dantz} instead of $l_1$-minimization. Finally, for the construction of $\hat g$ and $\hat f$,
we can use the direct approach of \cite{paper3}, which is given by
\begin{align*}
\hat g\colon\R\to\R,~t\mapsto f(t\hat a)\quad\text{and}\quad\hat f\colon B^d\to\R,~x\mapsto\hat g(\langle \hat a,x\rangle).
\end{align*}

Let us summarize this procedure as the following algorithm.
\vskip.2cm\noindent
\framebox[\textwidth][s]
{
\begin{minipage}{0.9\textwidth}
\vskip.2cm\centerline{{\bf Algorithm C}}\vskip-.10cm
\noindent\makebox[\linewidth]{\rule{10cm}{0.6pt}}
\begin{itemize}[leftmargin=*]
\item \emph{Input:} Ridge function $f(x)=g(\langle a,x\rangle)$ with \eqref{eq:assump10}, \eqref{eq:assump11}, \eqref{eq:assum2}-\eqref{eq:assum4}, $h,\sigma>0$
and $m\le d/(\log 6)^2$
\item Construct the $m\times d$ normalized Bernoulli matrix $\Phi$, c.f. \eqref{eq:Bern}, with rows denoted by $\varphi_1,\dots,\varphi_m\in\R^d$
\item Put $\displaystyle b_j=\frac{(f(h\ph_j)+\tilde z_j)-f(0)}{h}$, $j=1,\dots,m$
\item Put $\displaystyle \hat a=\frac{\Delta_{DS}(b)}{\|\Delta_{DS}(b)\|_{l_2^d}}$ for $\lambda_d=\sqrt{2\log d}$\\
\item Put $\displaystyle \hat g\colon\R\to\R,~t\mapsto f(t\hat a)$
\item Put $\displaystyle \hat f\colon B^d\to\R,~x\mapsto\hat g(\langle \hat a,x\rangle)$
\item \emph{Output:} $\hat f$
\end{itemize}
\vskip.2cm
\end{minipage}
}
\vskip.2cm

\begin{theo}
Let $f\colon B^d\to\R$ be a ridge function $f(x)=g(\langle a,x\rangle)$ with \eqref{eq:assump10}, \eqref{eq:assump11}, \eqref{eq:assum2}-\eqref{eq:assum4}.
Furthermore, let  $h,\sigma>0$ and let $s\le m\le d$ be positive integers with \eqref{eq:setting}.
Let $\tilde z_j\sim\mathcal{N}(0,\sigma^2)$ be independent. Then there is a constant $C_2>0$, such that
the function $\hat f$ defined by Algorithm C satisfies 
with high probability
\begin{align}\label{eq:unestf}
\|f-\hat f\|_\infty\leq2c_0\|a-\hat a\|_{l_2^d}\leq \frac{4 c_0\err(a,\hat a)}{g'(0)-\err(a,\hat a)},
\end{align}
where
\begin{align}\label{eq:noisehh}
\err(a,\hat a)&:=\left(\min\limits_{1\leq s_*\leq s}2 C_5 \log d \left(s_*\frac{\sigma^2}{h^2}+\tilde R^2 s_*^{-1}\right)\right)^{\frac 12} + C_7\frac{h}{\sqrt{s}},\\
\notag \tilde R&:= 2(R+2C_6h)
\end{align}
for some constants $C_5,C_6,C_7$. The second inequality in \eqref{eq:unestf} only holds if the denominator is positive.
\end{theo}
\begin{proof}
To prove this theorem, we apply Theorem \ref{random+deterministic} to \eqref{eq:defb}. Therefore we need to estimate the norm of $e\in\R^m$, defined by \eqref{eq:noise1}.
We obtain
\begin{align*}
\|e\|_{l_2^m}^2&
=\sum\limits_{j=1}^m\left[\langle a,\ph_j\rangle(g'(\xi_{h,j})-g'(0))\right]^2
	\leq\sum\limits_{j=1}^m\left(c_1h\langle a,\ph_j\rangle^2\right)^2\\
	&\leq c_1^2h^2\sum\limits_{j=1}^m\left(\|a\|_{l_1^d}\|\ph_j\|_{l_\infty^d}\right)^4
	\leq \frac{c_1^2h^2R^4}{m}
\end{align*}
and similarly we can show 
\begin{align*}
\|e\|_{l_\infty^m}\leq\frac{c_1hR^2}{m}.
\end{align*}
We can now apply Theorem \ref{random+deterministic} with $\eps=h R^2/\sqrt{m}$ to get
\begin{align*}
	\|\Delta_{DS}(b)-g'(0)a\|_{l_2^d}&\leq
	\left(\min\limits_{1\leq s_*\leq s}2 C_5 \log d \left(s_*\frac{2\sigma^2}{h^2}+\tilde R^2 s_*^{-1}\right)\right)^{\frac 12} + C_7\frac{h}{\sqrt{s}}\\
	&=:\err(a,\hat a)
\end{align*}
with $\tilde R= 2(R+2C_6h)$ and some constants $C_5,C_6,C_7$. And since we know that $a$ is $l_2^d$-normalized we set
\begin{align*}
	\hat a:=\frac{\Delta_{DS}(b)}{\|\Delta_{DS}(b)\|_{l_2^d}}.
\end{align*}
Applying Lemma \ref{stability-subspaces} we get
\begin{align*}
	\|a-\hat a\|_{l_2^d}
	\leq\frac{2 \err(a,\hat a)}{\|\Delta_{DS}(b)\|_{l_2^d}}
	\leq\frac{2 \err(a,\hat a)}{g'(0)\|a\|_{l_2^m}-\err(a,\hat a)}
	=\frac{2 \err(a,\hat a)}{g'(0)-\err(a,\hat a)},
\end{align*}
where the last inequality only holds if the denominator is positive. This proves the second inequality in \eqref{eq:unestf}.


The proof of the first part of \eqref{eq:unestf} proceeds as in \cite{paper3}.
First we define an approximation $\hat g$ to $g$
\begin{align}\label{eq:noise3}
	\hat g\colon\R\to\R,~t&\mapsto f(t\hat a).
\end{align}
This is indeed a good approximation to $g$ as for any $t\in[-1,1]$ we get
\begin{align}
\notag\vert \hat g(t)-g(t)\vert&=\vert g(\langle a,t\cdot\hat a\rangle)-g(t)\vert \leq c_0 \left\vert t \left(1-\langle a,\hat a\rangle\right)\right\vert=c_0 \vert t\vert\cdot\left\vert\langle a,a-\hat a\rangle\right\vert\\
\label{eq:noise2}&\leq c_0\cdot\|a-\hat a\|_{l_2^d}.
\end{align}
With this approximation $\hat g$ to $g$ we define the function $\hat f$ by
\begin{align*}
	\hat f\colon B^d\to\R,~x&\mapsto\hat g(\langle \hat a,x\rangle).
\end{align*}
It remains to show that $\hat f$ is a good approximation to $f$. With the help of \eqref{eq:noise3} and \eqref{eq:noise2} we obtain
\begin{align*}
	\vert f(x)-\hat f(x)\vert
	&=\vert g(\langle a,x\rangle)-\hat g(\langle\hat a,x\rangle)\vert\\
	&\leq \vert g(\langle\hat a,x\rangle)-\hat g(\langle\hat a,x\rangle)\vert+\vert g(\langle a,x\rangle)-g(\langle\hat a,x\rangle)\vert\\
	&\leq c_0\cdot\|a-\hat a\|_{l_2^d}+c_0\vert\langle a-\hat a,x\rangle\vert\\
	&\leq 2c_0\|a-\hat a\|_{l_2^d}
\end{align*}
for all $x\in B^d.$
\end{proof}

\section{Approximation of translated radial functions}

The methods we presented so far, as well as the methods of \cite{paper3}, were developed in the (quite restrictive) frame
of ridge functions. As an example of a possible extension of these algorithms, we consider the class of translated radial functions, i.e. functions of the form
\begin{align*}
	f\colon B^d\to\R,~x\mapsto f(x)=g(\|a-x\|_{l_2^d}^2)
\end{align*}
for some fixed $l_2^d$-normalized vector $a\in\R^d$
\begin{equation}\label{eq:anorm}
\|a\|_{l_2^d}=1
\end{equation}
and a function $g\colon[0,4]\to\R$. Hence, $f$ is constant on the spheres centered in $a$ or, equivalently,
it is a radial function translated by $a$.
Typically, we shall again assume that $g$ and $g'$ are Lipschitz continuous with constants $c_0$ and $c_1$, respectively. 

The idea to recover those functions is similar to the case of ridge functions. First we recover the center $a$ 
and then we define approximations $\hat g$ to $g$ and $\hat f$ to $f$.
\p
For a small constant $h>0$ and fixed vectors $x_i\in\R^d, i=1,\ldots,d$ we set
\begin{align*}
	\tilde a_i:=\frac{f(he_i+x_i)-f(x_i)}{h},
\end{align*}
where $e_1,\dots,e_d$ are again the canonical basis vectors of $\R^d$.
With help of the mean value theorem we can express this as
\begin{align*}
	\tilde a_i&=\frac{f(he_i+x_i)-f(x_i)}{h}
	=\frac{g(\|a-he_i-x_i\|_{l_2^d}^2)-g(\|a-x_i\|_{l_2^d}^2)}{h}\\
	&=g'(\xi_{h,i})\frac{\|a-he_i-x_i\|_{l_2^d}^2-\|a-x_i\|_{l_2^d}^2}{h}
\end{align*}
for some real $\xi_{h,i}$ between $\|a-he_i-x_i\|_{l_2^d}^2$ and $\|a-x_i\|_{l_2^d}^2$. The nominator can be simplified by
\begin{align*}
	\|a-he_i-x_i\|_{l_2^d}^2-\|a-x_i\|_{l_2^d}^2
	&=\langle a-he_i-x_i,a-he_i-x_i\rangle-\langle a-x_i,a-x_i\rangle\\
	&=-2h\langle a,e_i\rangle+h^2\langle e_i,e_i\rangle+2h\langle e_i,x_i\rangle\\
	&=-2ha_i+h^2+2h x_{i,i}.
\end{align*}
Let us choose $x_i=-(h/2)e_i$ to get
\begin{align*}
	\tilde a_i
	&=\frac{f((h/2)e_i)-f(-(h/2)e_i)}{h}\\
	&=-2g'(\xi_{h,i})a_i
\end{align*}
for some $\xi_{h,i}$ between $\|a-(h/2)e_i\|_{l_2^d}^2$ and $\|a+(h/2)e_i\|_{l_2^d}^2$.
Next let us note that $\xi_{h,i}$ is very close to $1=\|a\|_{l_2^d}^2$:
\begin{align}
\notag\left\vert\xi_{h,i}-1\right\vert
	&\le \max\left\{\left\vert 1-\|a-(h/2)e_i\|_{l_2^d}^2\right\vert~,~\left\vert1-\|a+(h/2)e_i\|_{l_2^d}^2\right\vert\right\}\\
\label{eq:xi}&=\max\left\{\left\vert -ha_i-h^2/4\right\vert~,~\left\vert ha_i-h^2/4\right\vert\right\}\\
\notag&\leq h+h^2/4.
\end{align}
Finally we obtain that $\hat a$ is a good approximation to $-2g'(\|a\|_{l_2^d}^2) a=-2g'(1)a$, since
\begin{align*}
	\|\tilde a+2g'(1)a\|_{l_2^d}^2
	&=\sum\limits_{i=1}^d\left(-2g'(\xi_{h,i})a_i+2g'(1)a_i\right)^2\\
	&=4\sum\limits_{i=1}^d\left(\left(g'(\xi_{h,i})-g'(1)\right)a_i\right)^2
	\leq 4\sum\limits_{i=1}^d\left(c_1\left\vert \xi_{h,i}-1\right\vert a_i\right)^2\\
	&\leq 4c_1^2\sum\limits_{i=1}^d\left(\left(h+h^2/4\right) a_i\right)^2
	= 4c_1^2\left(h+h^2/4\right)^2\sum\limits_{i=1}^d a_i^2\\
	&=4c_1^2\left(h+h^2/4\right)^2.
\end{align*}
Thus $\tilde a$ is almost a multiple of $a$. Again, we need to assume that the derivative of $g'$ is non-trivial in some sense.
Due to the construction, we replace \eqref{eq:assum4} by the condition
\[
g'(1)\not =0.
\]
Then the $l_2^d$-normalized vector
\begin{align*}
\hat a:=\frac{\tilde a}{\|\tilde a\|_{l_2^d}}
\end{align*}
approximates $a$, possibly up to a sign. Choosing any vector $\hat a^\perp\in\R^d$ orthogonal to $\hat a$, we can identify the sign by sampling along $\hat a^\perp.$
Afterwards, the correct sign might be assigned to $\hat a.$ We will therefore restrict ourselves to the case
\begin{equation}\label{eq:g'1}
g'(1)>0.
\end{equation}

Once an approximation of $a$ was recovered, it is again easy to define an approximation of $g$, and finally of $f$. We summarize this procedure as the following algorithm.
\vskip.2cm\noindent
\framebox[\textwidth][s]
{
\begin{minipage}{0.9\textwidth}
\vskip.2cm\centerline{{\bf Algorithm D}}\vskip-.10cm
\noindent\makebox[\linewidth]{\rule{10cm}{0.6pt}}

\begin{itemize}[leftmargin=*]
\item \emph{Input:} Translated radial function $f\colon B^d\to\R$ with $f(x)=g(\|a-x\|_{l_2^d}^2)$, $a$ and $g$ with \eqref{eq:anorm}, \eqref{eq:assum2}, \eqref{eq:assum3} and \eqref{eq:g'1}
\item Put $\tilde a_i:=\frac{f(he_i/2)-f(-he_i/2)}{h}$, $i=1,\dots,d$
\item Put $\hat a:=\frac{\tilde a}{\|\tilde a\|_{l_2^d}}$
\item Put $\hat g\colon[0,4]\to\R,~t\mapsto f(\hat a(1-\sqrt{t}))$
\item Put $\hat f\colon B^d\to\R,~x\mapsto\hat g(\|\hat a-x\|_{l_2^d}^2)$
\item \emph{Output:} $\hat f$
\end{itemize}
\vskip.2cm
\end{minipage}
}
\vskip.2cm
The performance of Algorithm D is estimated by the following Theorem.
\begin{theo}\label{thm:sing}
Let $f\colon B^d\to\R$, $g\colon[0,4]\to\R$ and $a\in\R^d$ be such that $f(x)=g(\|a-x\|_{l_2^d}^2)$ and $a$ and 
and $g$ satisfy \eqref{eq:anorm},\eqref{eq:assum2}, \eqref{eq:assum3} and \eqref{eq:g'1}.
Then 
\begin{align}
\label{eq:ffhat}\|f-\hat f\|_{\infty}&\leq c_0\left(2\|\hat a- a\|_{l_2^d}+\|\hat a- a\|_{l_2^d}^2\right)\\
\intertext{and}
\label{eq:aahat}\|\hat a- a\|_{l_2^d}&\le \frac{2 c_1\left(h+h^2/4\right)}{g'(1) -c_1(h+h^2/4)}
\end{align}
if $g'(1) -c_1(h+h^2/4)$ is positive.
\end{theo}
\begin{proof}
First, we estimate the difference between $a$ and $\hat a$. By \eqref{eq:xi} and \eqref{eq:assum3}
\begin{align*}
	g'(1)-\vert g'(\xi_{h,i})\vert
	\leq \vert g'(1)-g'(\xi_{h,i})\vert
	\leq c_1\vert 1-\xi_{h,i}\vert
	\leq c_1(h+h^2/4),
\end{align*}
hence
\begin{align}\label{eq:xi2}
\vert g'(\xi_{h,i})\vert
	\geq g'(1) -c_1(h+h^2/4).
\end{align}
Therefore, if the right hand side of \eqref{eq:xi2} is positive, we get
\begin{align*}
	\|\tilde a\|_{l_2^d}^2
	&=\sum\limits_{i=1}^d\vert\tilde a_i\vert^2
	=4\sum\limits_{i=1}^d\vert g'(\xi_{h,i})a_i\vert^2\\
	&\geq 4\sum\limits_{i=1}^d\left( g'(1) -c_1(h+h^2/4)\right)^2a_i^2\\
	&=4\left(g'(1) -c_1(h+h^2/4)\right)^2.
\end{align*}
Now we apply Lemma \ref{stability-subspaces} to obtain
\begin{align}\label{eq:sing2}
\|\hat a-a\|_{l_2^d}
\leq\frac{4c_1(h+h^2/4)}{\|\tilde a\|_{l_2^d}}
\leq\frac{2c_1(h+h^2/4)}{g'(1) -c_1(h+h^2/4)}.
\end{align}
Given the approximation $\hat a$ to $a$ we define an approximation $\hat g$ to $g$ by
\begin{align*}
\hat g\colon[0,4]\to\R,~t\mapsto f(\hat a(1-\sqrt{t})).
\end{align*}
Essentially, $\hat g$ is the restriction of $f$ onto $\{\lambda\hat a:\lambda\in\R\}\cap B^d$. Using \eqref{eq:assum2} we obtain the estimate
\begin{align}
\notag \vert g(t)-\hat g(t)\vert
&=\left\vert g(t)-g(\|a-\hat a+\sqrt{t}\hat a\|_{l_2^d}^2)\right\vert\leq c_0\left\vert t-\|a-\hat a+\sqrt{t}\hat a\|_{l_2^d}^2\right\vert\\
\notag &=c_0\left\vert 2\sqrt{t}\langle a-\hat a,\hat a\rangle+\|a-\hat a\|_{l_2^d}^2\right\vert=c_0\left\vert 2(1-\sqrt{t})(1-\langle a,\hat a\rangle) \right\vert\\
\label{eq:sing1}&=c_0\,\left\vert 1-\sqrt{t}\right\vert \cdot \|a-\hat a\|_{l_2^d}^2\le c_0\,\|a-\hat a\|_{l_2^d}^2 
\end{align}
for all $t\in[0,4]$. Next we define
\begin{align*}
	\hat f\colon B^d\to\R,~x\mapsto \hat g(\|\hat a-x\|_{l_2^d}^2).
\end{align*}
With \eqref{eq:sing2} and \eqref{eq:sing1} we get the final estimate
\begin{align*}
	\vert f(x)-\hat f(x)\vert
	&=\left\vert g(\|a-x\|_{l_2^d}^2)-\hat g(\|\hat a-x\|_{l_2^d}^2)\right\vert\\
	&\leq \left\vert g(\|a-x\|_{l_2^d}^2)-g(\|\hat a-x\|_{l_2^d}^2)\right\vert+\left\vert g(\|\hat a-x\|_{l_2^d}^2)-\hat g(\|\hat a-x\|_{l_2^d}^2)\right\vert\\
	&\leq c_0\left\vert \|a-x\|_{l_2^d}^2-\|\hat a-x\|_{l_2^d}^2\right\vert+c_0 \, \|a-\hat a\|_{l_2^d}^2\\
	&=2c_0\left\vert \langle a-\hat a,x\rangle\right\vert +c_0\, \|a-\hat a\|_{l_2^d}^2\\
	&\leq c_0\left(2\|a-\hat a\|_{l_2^d}+\|a-\hat a\|_{l_2^d}^2\right).
\end{align*}
\end{proof}
\begin{remark}\label{rem:thm:sing} (Extensions of Theorem \ref{thm:sing})
\begin{enumerate}
\item[(i)] We assumed in Theorem \ref{thm:sing}, that the function $g$ and its derivative $g'$ are both Lipschitz.
If we assume this property only on the interval $(1-(h+h^2/4),1+(h+h^2/4))$,
we can still recover at least \eqref{eq:sing2}. This applies even to the case, when $g$ (and also its derivative) are unbounded near the origin.
In that case, we can still  approximate the position of the singularity, although uniform approximation of $f$ is out of reach.
\item[(ii)]	As in the approximation scheme for ridge functions we can use techniques from compressed sensing to recover $f$ if $a$ is compressible.
To be more precise, if $a$ satisfies $\|a\|_{l_1^d}\leq R$ and $\Phi\in\R^{m\times d}$ 
is a normalized Bernoulli matrix with rows $\ph_1,\ldots,\ph_m\in\R^d$ we define
\begin{align*}
\tilde b_j&:=\frac{f((h/2)\ph_j)-f(-(h/2)\ph_j)}{h}.
\end{align*}
As $f$ is defined only on the unit ball $B^d$ and $\|\ph_j\|_{l_2^d}=\sqrt{d/m}$, we must always have at least
$h\le 2\sqrt{m/d}$ to ensure that $(h/2)\ph_j\in B^d.$ To allow for comparison with the non-compressible case just discussed in Theorem \ref{thm:sing}, we denote
\[
\tilde h=h/2\cdot\sqrt{d/m},
\]
which leads to
\begin{align}\label{eq:b}
\tilde b_j&=\frac{\displaystyle f\Bigl(\tilde h\frac{\ph_j}{\|\ph_j\|_{l_2^d}}\Bigr)-f\Bigl(-\tilde h\frac{\ph_j}{\|\ph_j\|_{l_2^d}}\Bigr)}{h}.
\end{align}

By defining the deterministic noise $e\in\R^m$
\begin{align}\label{eq:b2}
	\tilde b=-2g'(1)\Phi a+e
\end{align}
we can show with similar calculations as before that
\begin{align}\label{eq:singcs1}
\|e\|_{l_2^m}&\leq\eta:=2c_1R\left(\frac{2R\tilde h}{\sqrt{d}}+{\tilde h}^2\right).
\end{align}
Using the $(P_{1,\eta})$ minimizer of \cite{C} we put
\[
\tilde a=\argmin_{z\in\R^d}\|z\|_{l_1^d}\quad \text{s.t.}\quad \|\Phi z-\tilde b\|_{l_2^m}\le \eta
\]
with $\eta$ given by the right hand side of \eqref{eq:singcs1}.  We then get the estimate, cf. \cite[Theorem 4.22]{paper7} or \cite[Theorem 1.6]{BCKV},
\[
\|\tilde a-2g'(1)a\|_{l_2^d}\le \varrho:=\frac{C\sigma_s(2g'(1)a)_1}{\sqrt{s}}+D\eta
\]
with two universal constants $C,D>0.$ Here again $s\leq C_2m/\log(d/m)$.
Lemma \ref{stability-subspaces} (with $l_2^d$ instead of $l_1^d$) gives for $\hat a=\tilde a/\|\tilde a\|_{l_2^d}$
\[
\|\hat a-a\|_{l_2^d}\le 2\varrho/\|\tilde a\|_{l_2^d}.
\]
Finally, using
\[
\|\tilde a\|_{l_2^d}\ge 2g'(1)\|a\|_{l_2^d}-\|\tilde a-2g'(1)a\|_{l_2^d}\ge 2g'(1)-\varrho,
\]
we get
\[
\|\hat a-a\|_{l_2^d}\le \frac{2\varrho}{2g'(1)-\varrho}
\]
if $2g'(1)>\varrho.$

This gives a replacement of \eqref{eq:sing2}, the rest of the proof of Theorem \ref{thm:sing} then applies without further modifications.
\item[(iii)] Once we have this approximation scheme using techniques from compressed sensing, we can easily extend it to an approximation scheme with noisy measurements.
We assume again that $\tilde b$ from \eqref{eq:b} is corrupted by noise $z/h$, where the components of $z=(z_1,\dots,z_m)^T$ are again independent ${\mathcal N}(0,\sigma^2)$
distributed random variables. Formula \eqref{eq:b2} is then replaced by $\tilde b=-2g'(1)\Phi a+e+z/h$ and Dantzig selector can be applied.
\end{enumerate}
\end{remark}

\section{Numerical results}

In this section we investigate the performance of the algorithms presented so far in several model situations. The results shed a new light on some of the aspects, which we did not discuss in detail, especially on the size
of the constants used in previous theorems. All the approximation schemes started by looking for a good approximation $\hat a$ of the unknown direction $a$ and, consequently, the quality of the uniform approximation
of $f$ by $\hat f$ was then bounded by the corresponding distance between $\hat a$ and $a$. In what follows, we will therefore discuss only the approximation error between $a$ and $\hat a$.

\subsection{Ridge functions on cubes}

\begin{figure}
     \centering
     \includegraphics[width=6.1cm]{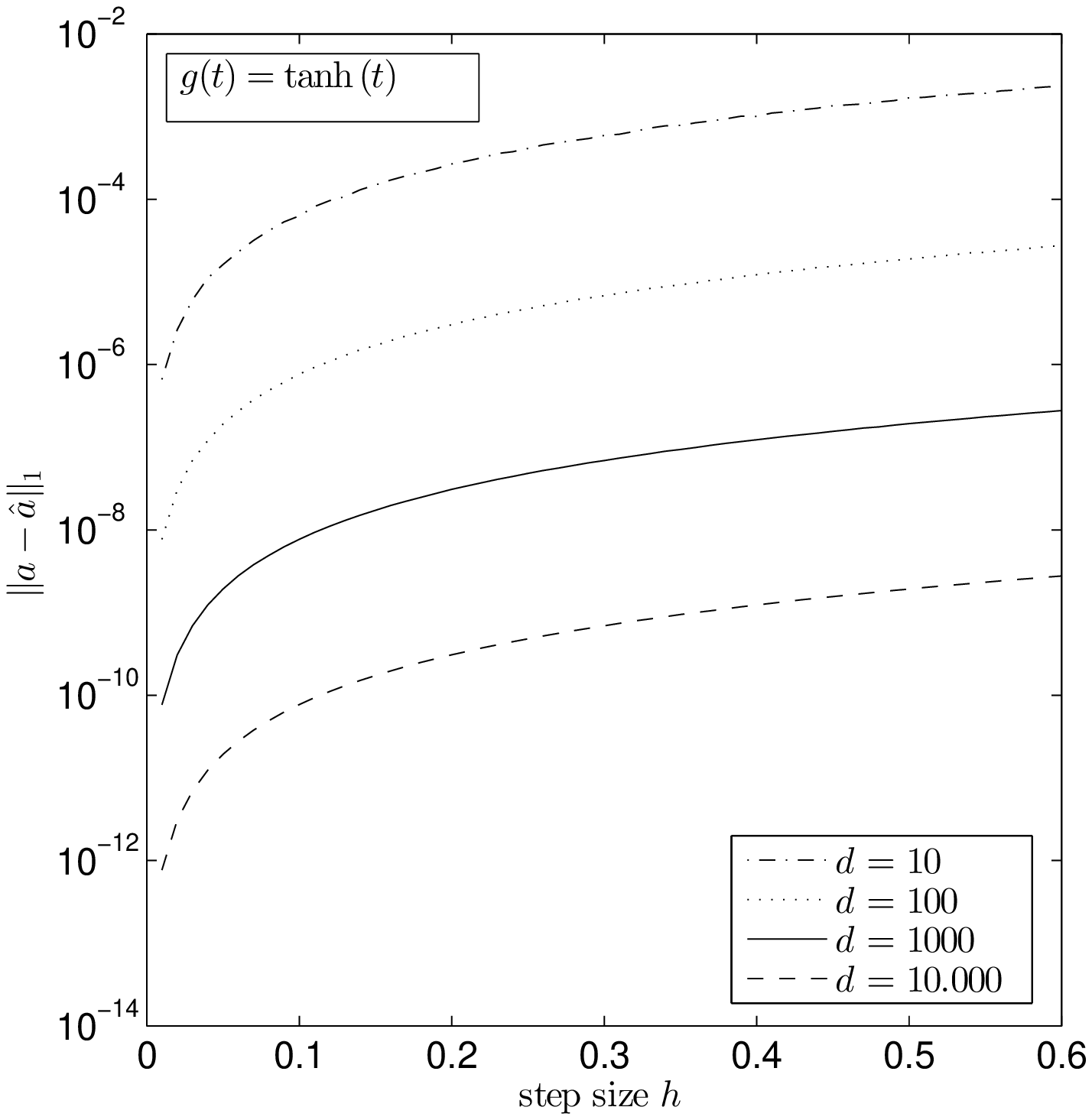}\quad
      \includegraphics[width=6.1cm]{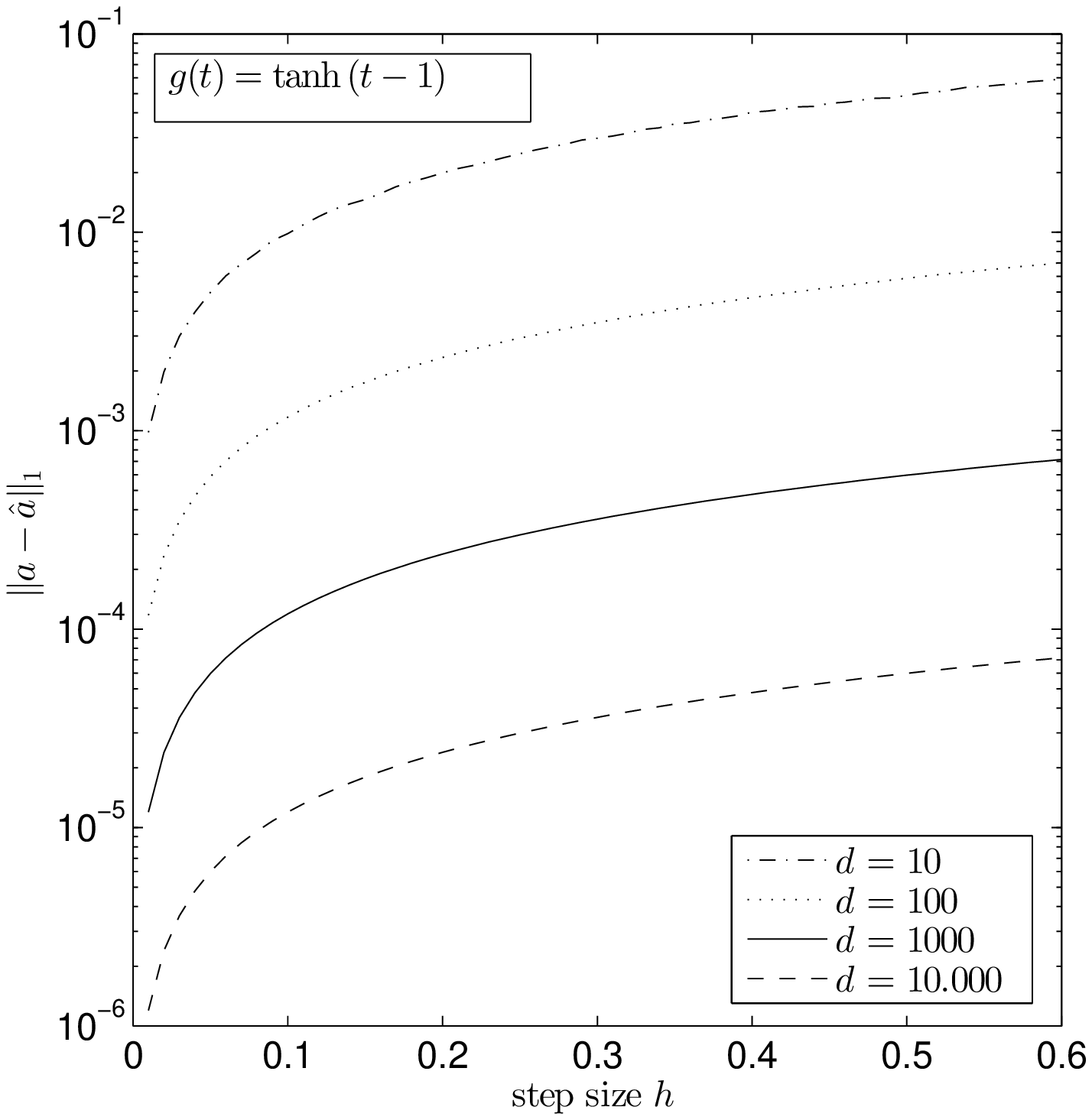}
     \caption{Approximation of $a$ according to Algorithm A with $g(t)=\tanh(t)$ (left) and $g(t)=\tanh(t-1)$ (right).}
     \label{fig:approx-A}
\end{figure}


We start with Algorithm A, i.e. with approximation of a  ridge function $f(x)=g(\langle a,x\rangle)$ defined on the cube $[-1,1]^d$ with $\|a\|_{\ell_1^d}=1.$
We have considered different dimensions ($d\in\{10,100,1000,10.000\}$).
As the Algorithm A does not make any use of sparsity of $a$, it is reasonable to assume, that all its coordinates are equally likely to be non-zero.
The entries of $a$ were therefore always independently normally distributed (i.e. $a_i\sim\mathcal{N}(0,1)$), afterwards $a$ got $l_1^d$-normalized according to \eqref{eq:assum1}.

Figure \ref{fig:approx-A} shows the (average) approximation error $\|a-\hat a\|_{\ell_1^d}$ in dependence of the step size $h>0$
for two different profiles $g(t)=\tanh(t)$ and $g(t)=\tanh(t-1)$. Note that the $y$-axis scales logarithmically.
Let us give some remarks on Figure \ref{fig:approx-A}.
\begin{itemize}
\item The approximation improves rapidly with growing dimension. This is given by considering the non-sparse ridge vectors $a$ and by the concentration of measure phenomenon
as described also in Remark \ref{rem:3.4}.
\item Smaller step size $h$ implies also better quality of approximation, but already reasonable sizes of $h$ (i.e. $h=0.2$) imply relatively very small errors.
\item Finally, the second derivative of the first profile at zero vanishes, were it is non-zero for the second profile. Therefore, the first order differences approximate
the first order derivative less accurately in that case, leading to larger (but still surprisingly small) approximation errors.
\end{itemize}


\begin{figure}[h]
	\centering
	\includegraphics[width=6.1cm]{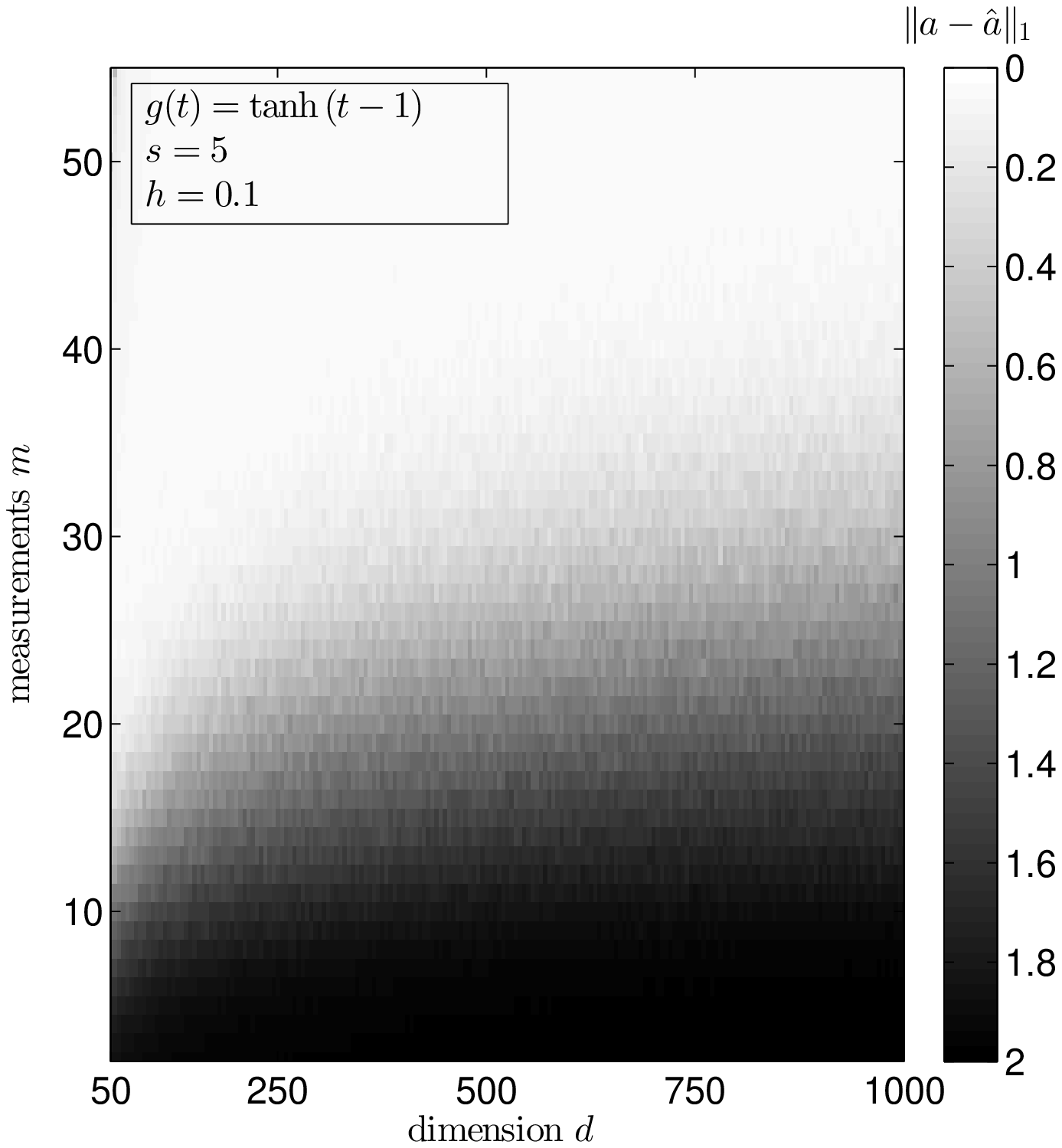}\includegraphics[width=6.1cm]{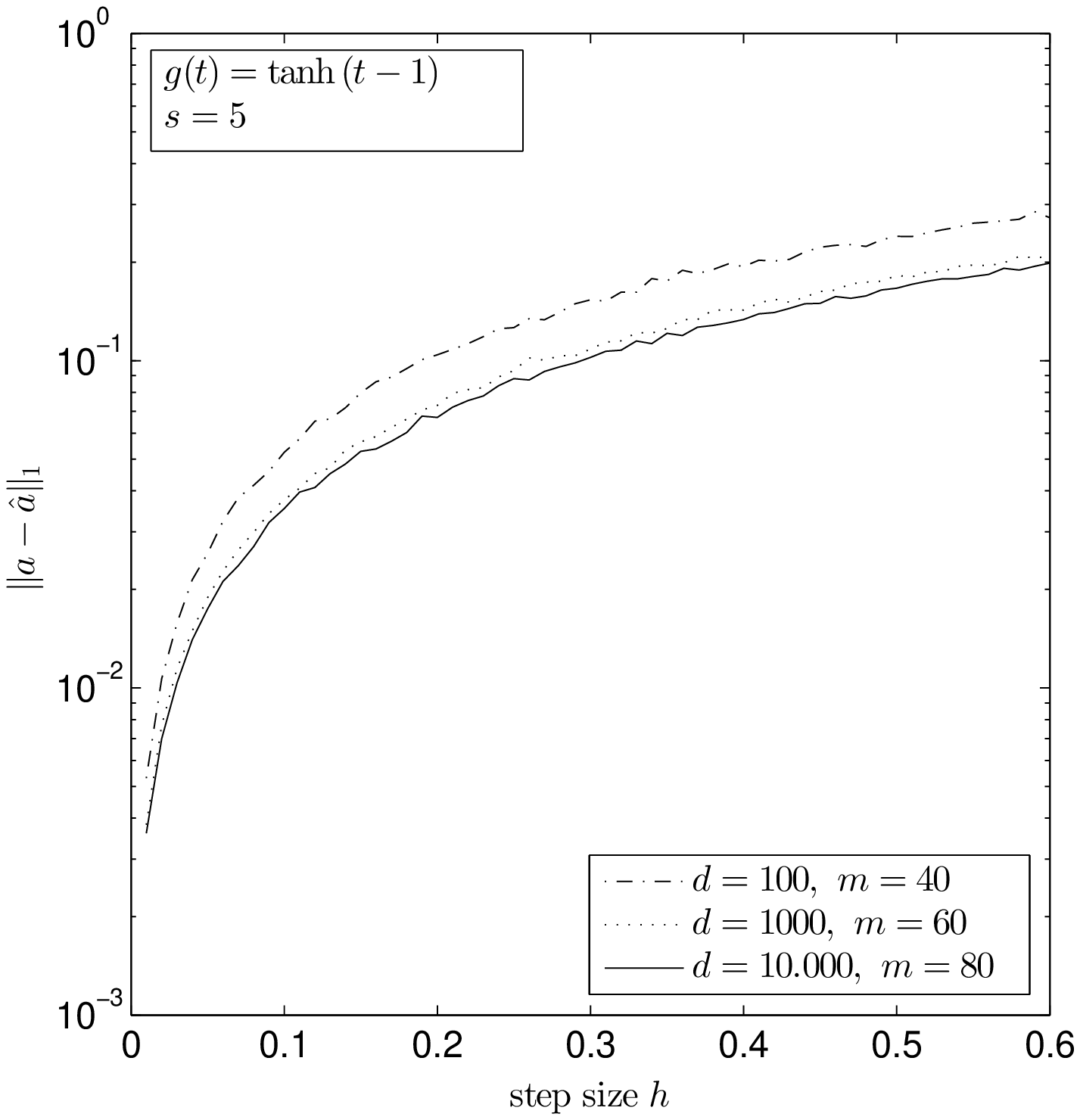}
	\caption{Face transition for the approximation of $a$ and the average error of $\|a-\hat a\|_{\ell_1^d}$ according to algorithm B.}
	\label{fig:approx-B}
\end{figure}

The left part of Figure \ref{fig:approx-B} shows the dependence of the number of the sampling points $m$ on the dimension $d$ and sparsity $s$, cf. \eqref{eq:setting}, when using Algorithm B.
We fixed the ridge profile $g(t)=\tanh(t-1)$, the sparsity $s=5$ and the step size $h=0.1$ and constructed an $s$-sparse random vector $a$ by {\sc Matlab} command {\tt sprandn}, followed by the $\ell_1^d$-normalization.
For each integer $d$ between $50$ and $1000$ and for each integer $m$ between $1$ and $55$, we then run the Algorithm B 120 times and the average approximation error $\|a-\hat a\|_{\ell_1^d}$
corresponds afterwards to the shade of grey of the point with coordinates $d$ and $m$. In accordance with the theory of compressed sensing (and with Remark \ref{rem:AlgB}), we observe that
the number of measurements needs to grow only logarithmically in the dimension $d$ to guarantee good approximation with high probability. The right part of Figure \ref{fig:approx-B}
then shows the average value of $\|a-\hat a\|_{\ell_1^d}$ for the same profile and sparsity for three different pairs of $(d,m)$. We observe, that especially for large dimensions even extremely small number of measurements guarantees
already reasonable approximation errors.

\subsection{Noisy measurements}

\begin{figure}
	\centering
	\includegraphics[width=6cm]{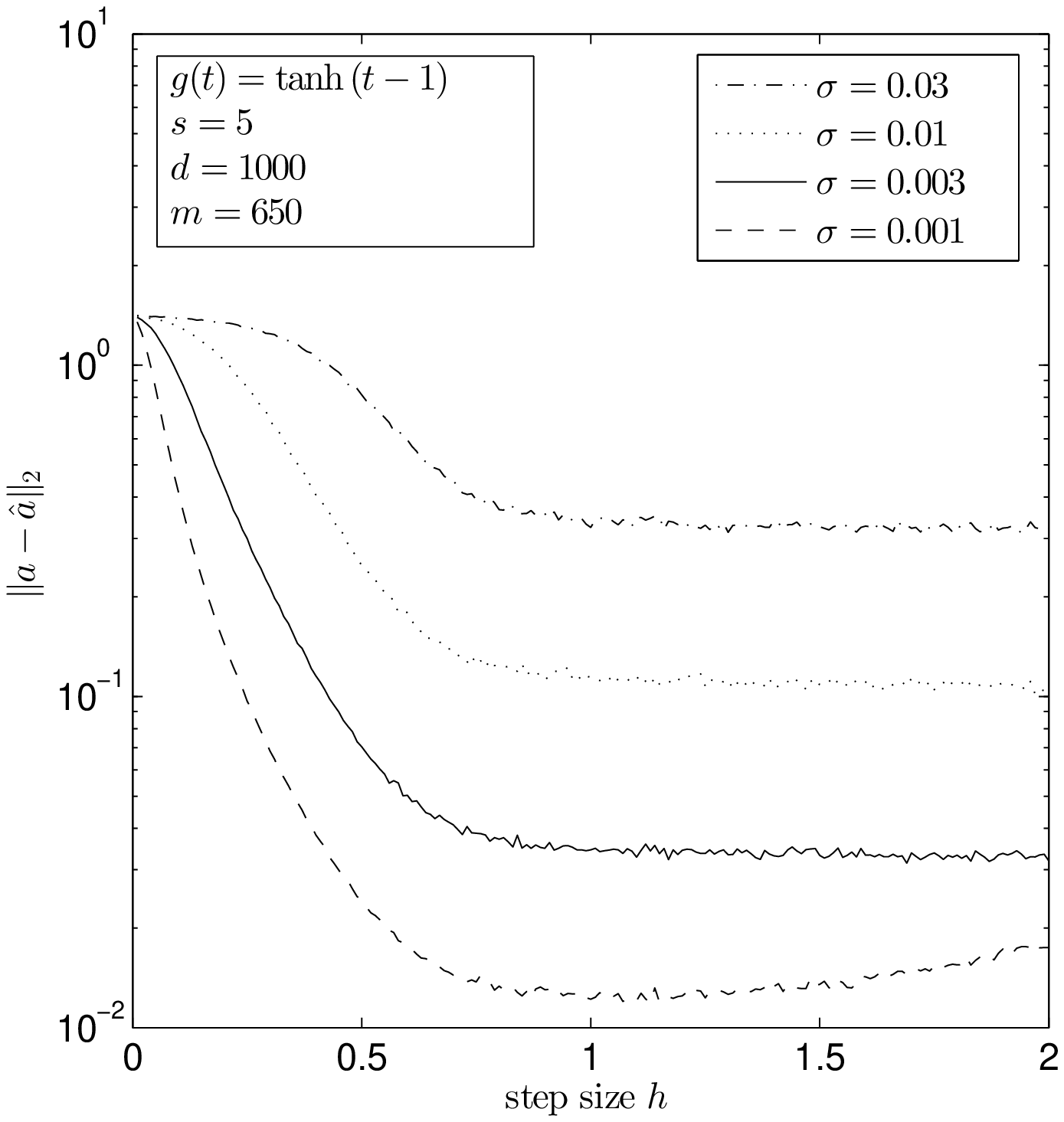}\quad
	\includegraphics[width=6cm]{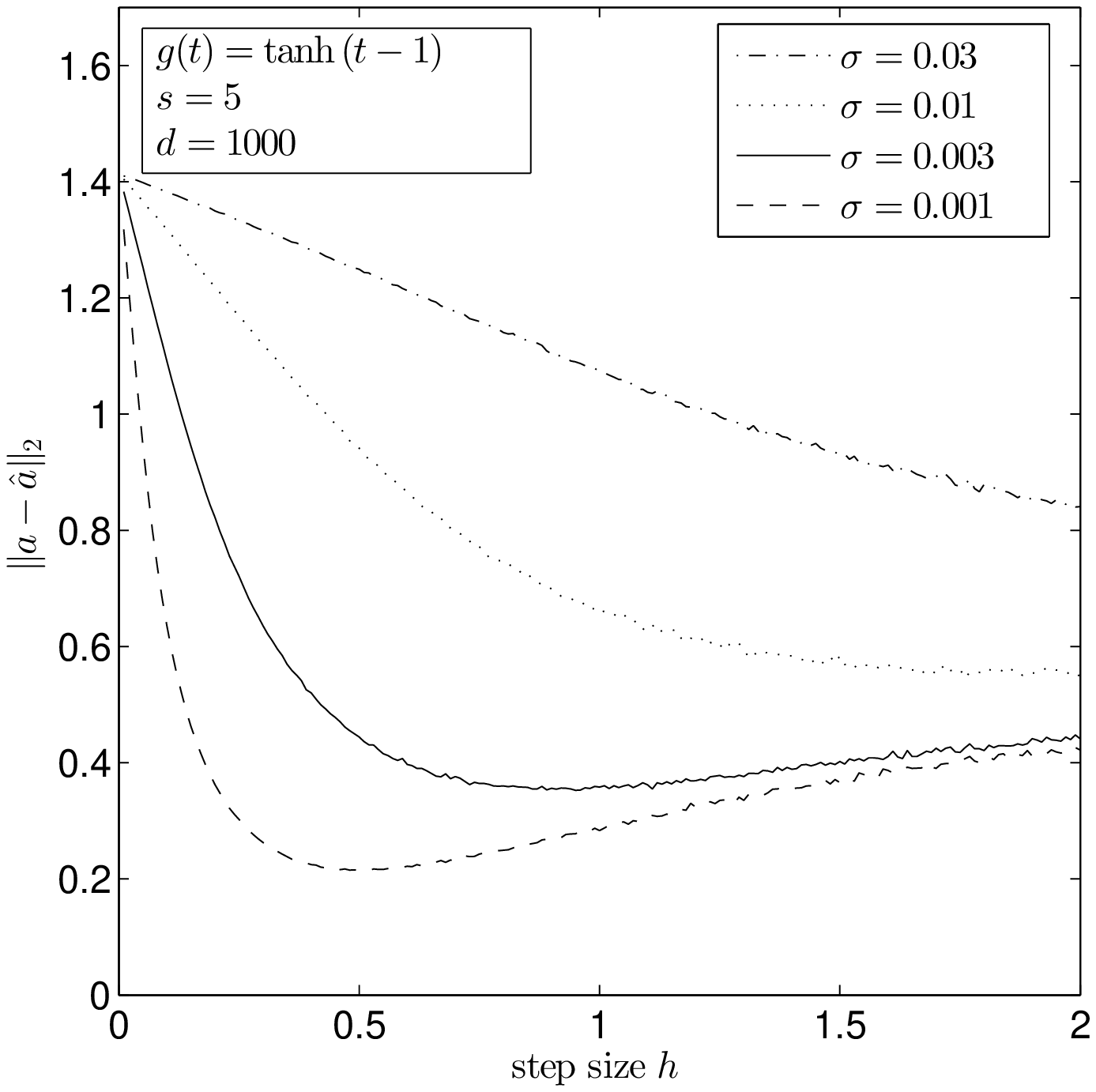}
	\caption{Approximation of $a$ with noisy measurements according to Algorithm C (left) and a modification of Algorithm A (right). Note, that only the $y$-axis of the left plot is logarithmic.}
	\label{fig:approx-C}
\end{figure}

Figure \ref{fig:approx-C} studies the performance of the recovery of the ridge vector $a$ from noisy measurements as described in Algorithm C.
We fixed the parameters $d=1000$, $m=400$, and $s=5$, the ridge profile $g(t)=\tanh(t-1)$ and four different noise levels $\sigma\in\{0.03, 0.01, 0.003, 0.001\}$.
We have used the $\ell_1$-MAGIC implementation of Dantzig selector, available at the web page of Justin Romberg at {\tt http://users.ece.gatech.edu/\textasciitilde justin/l1magic/}.
As the noise level gets amplified by the factor $1/h$, when taking the first order differences, cf. \eqref{eq:noiseh}, it is not surprising that the recovery fails
completely for small values of $h$. On the other hand, for large values of $h$, the correspondence between first order differences and first order derivatives
gets weaker and the quality of approximation deteriorates as well. This effect is clearly visible from \eqref{eq:noisehh} and, numerically, in the left part of Figure \ref{fig:approx-C},
where there is an optimal $h$ for the recovery of $a$. Strictly speaking, the functions considered in Section \ref{stability-ridge} were defined only on the unit ball $B^d$,
so that the value of $h$ in Figure \ref{fig:approx-C} should be limited to be smaller than $\sqrt{m/d}$. We have decided to include also larger values $h$ to exhibit
the optimal $h$, although for our profile and our parameters it lies outside of this interval.

Although not discussed before, it is quite straightforward to modify the non-probabilistic Algorithm A also to the case of noisy measurements. Essentially, the gradient $\nabla f(0)$
is then approximated by the first-order differences, this time corrupted by noise. We applied this approach to the profile $g(t)=\tanh(t-1)$ and parameters just described with the
results plotted in the right part of Figure \ref{fig:approx-C}. We observe that the approximation errors get much larger, demonstrating once again the success of Dantzig selector.

\subsection{Shifted radial functions}

\begin{figure}
\centering
\includegraphics[width=6.1cm, height=6cm]{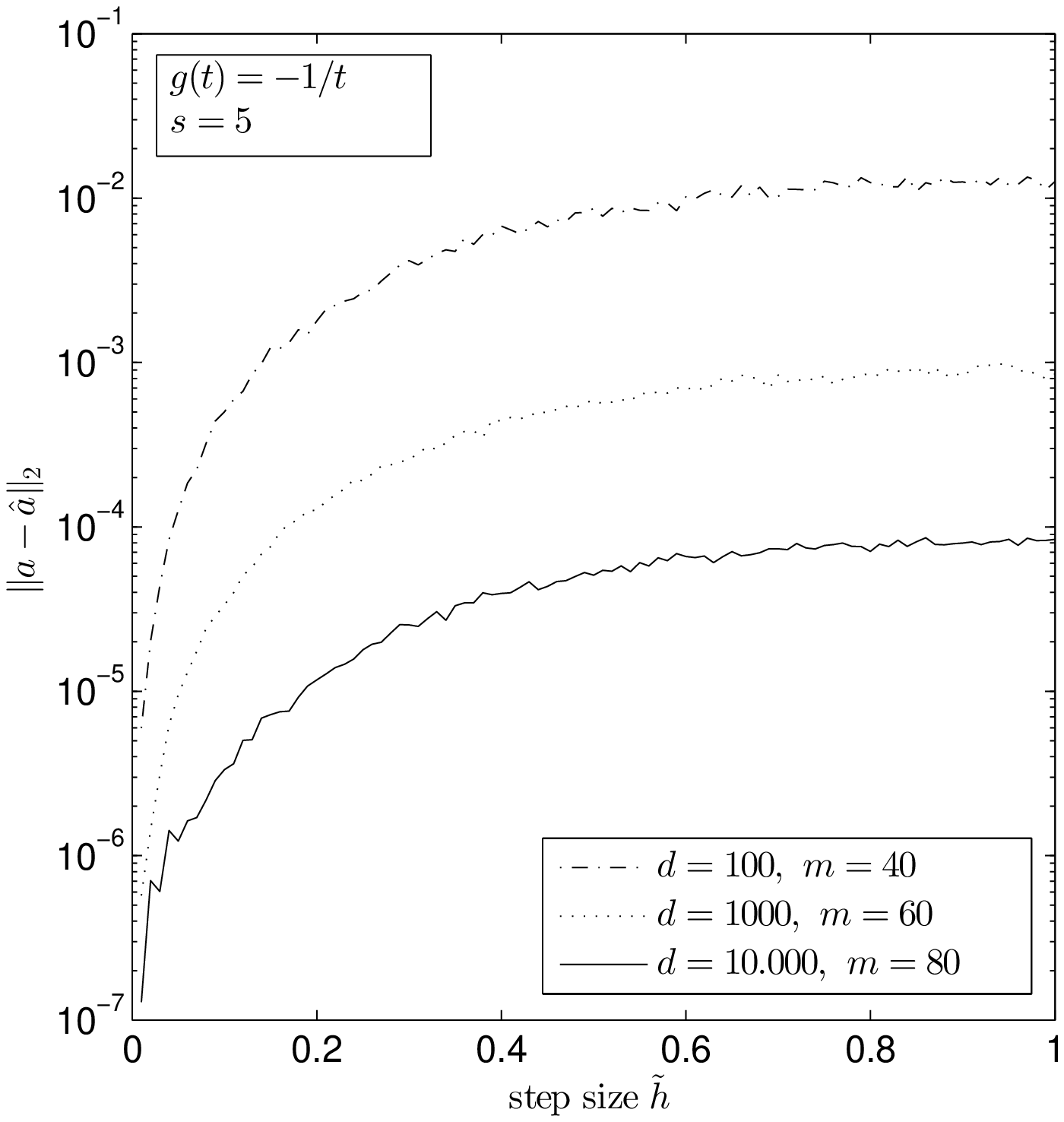}\quad
\includegraphics[width=6.1cm, height=6cm]{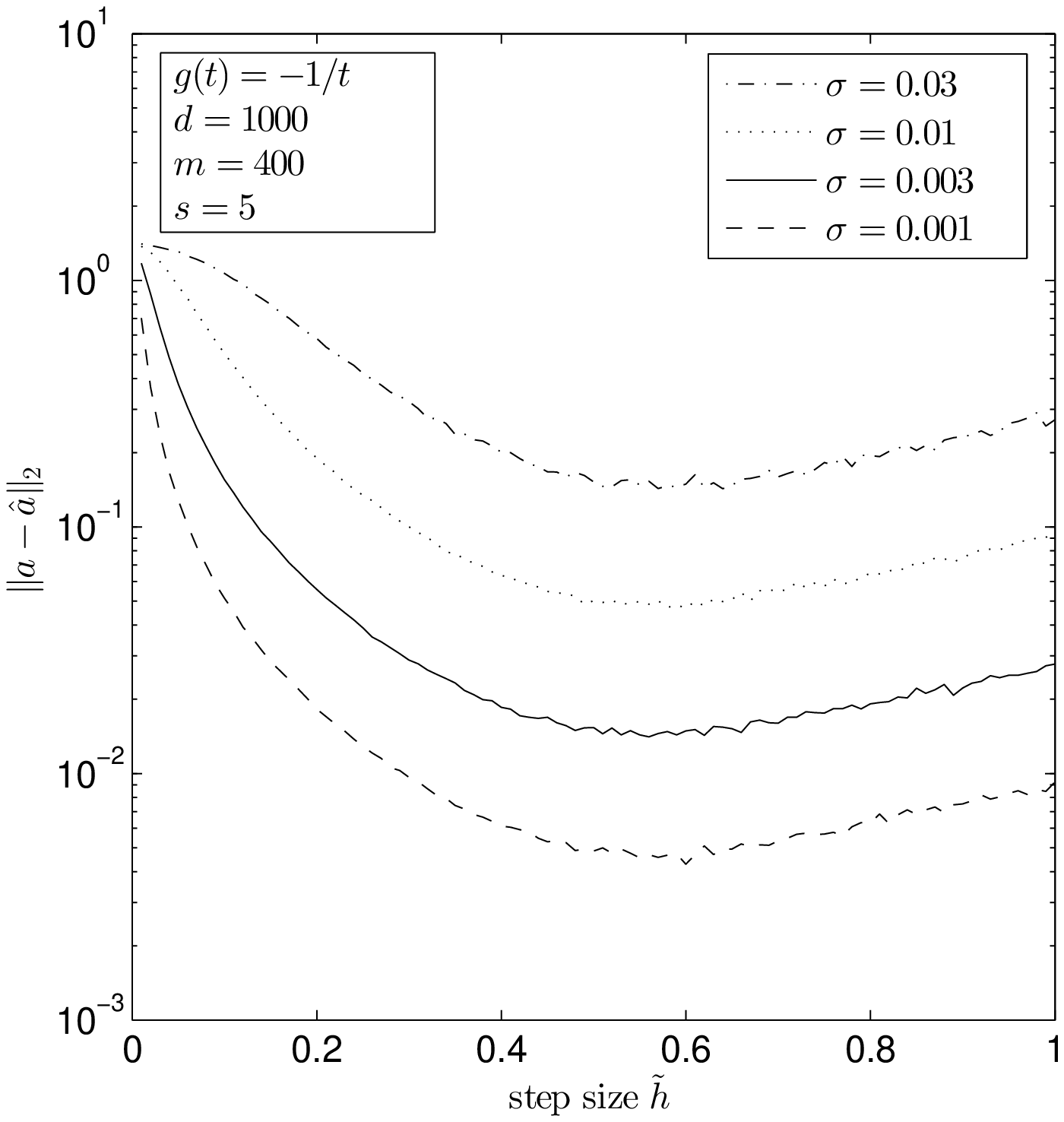}
\caption{Approximation of $a$ according to algorithm D with sparsity (left) and with noisy measurements (right).}
\label{fig:approx-D}
\end{figure}

In Figure \ref{fig:approx-D} we considered the approximation of the pole $a$ of a shifted radial function $f$ with $f(x)=g(\|a-x\|_{l_2^d}^2)$ and $g(t)=-1/t$.
On the left plot, we fixed the sparsity $s=5$ and considered three values of $d=100, d=1000$ and $d=10.000$. The number of measurements was then $m=40, m=60$, or $m=80$, respectively.
Finally, we run the modification of Algorithm D described in Remark \ref{rem:thm:sing} and plot the average approximation error $\|a-\hat a\|_{\ell_2^d}$ against the step size $h$.
The right hand plot of Figure \ref{fig:approx-D} shows the noise-aware modification of Algorithm D described also in Remark \ref{rem:thm:sing}.
%
%

\end{document}